\documentclass[12pt]{amsart}
\usepackage{amscd, amssymb, amsfonts, amsthm, mathrsfs}
\usepackage{fullpage}
\usepackage{latexsym}
\usepackage[mathscr]{}
\usepackage{verbatim}
\usepackage{enumerate}
\usepackage[all,cmtip]{xy}
\usepackage[colorlinks=true]{hyperref}
\newcommand{\iso}{\cong}
\numberwithin{equation}{section}

\newtheorem{thm}[equation]{Theorem}

\newtheorem{lem}[equation]{Lemma}
\newtheorem{cor}[equation]{Corollary}
\newtheorem{prop}[equation]{Proposition}

\theoremstyle{definition}
\newtheorem{defn}[equation]{Definition}
\newtheorem{ex}[equation]{Example}

\newtheorem{notn}[equation]{Notation}
\theoremstyle{remark}
\newtheorem{rem}[equation]{Remark}
\theoremstyle{remark}
\newtheorem{rems}[equation]{Remarks}

\DeclareMathOperator{\Aut}{Aut}
\DeclareMathOperator{\GL}{GL}
\DeclareMathOperator{\End}{End}
\DeclareMathOperator{\SL}{SL}
\DeclareMathOperator{\Sp}{Sp}

\DeclareMathOperator{\PGL}{PGL}

\DeclareMathOperator{\Lie}{Lie}
\DeclareMathOperator{\Ad}{Ad}
\DeclareMathOperator{\rk}{rk}
\DeclareMathOperator{\id}{id}
\DeclareMathOperator{\hgt}{ht}
\DeclareMathOperator{\Hom}{Hom}

\newcommand\inverse{{^{-1}}}

\renewcommand{\Im}{{\rm Im}}

\newcommand{\Gm}{\mathbb{G}_m}
\renewcommand{\AA}{\mathscr{A}}
 
\newcommand{\ZZ}{\mathbb{Z}}
\newcommand{\BBF}{\mathbb{F}}

\newcommand{\BBN}{\mathbb{N}}
\newcommand{\BBG}{\mathbb{G}}
\newcommand{\NN}{\mathscr{N}}
\newcommand{\UU}{\mathscr{U}}
\newcommand{\DD}{\mathscr{D}}

\renewcommand{\SS}{\mathscr{S}}

\newcommand{\frakg}{\mathfrak g}
\newcommand{\frakp}{\mathfrak p}


\dedicatory{Dedicated to the fond memory of Gary Seitz}

\title[On good $A_1$ subgroups and overgroups of distinguished unipotent elements]
{On good $A_1$ subgroups, Springer maps, and overgroups of distinguished unipotent elements in reductive groups}

\author[M.\  Bate]{Michael Bate}
\address%[M.\  Bate]
{Department of Mathematics,
University of York,
York YO10 5DD,
United Kingdom}
\email{michael.bate@york.ac.uk}

\author[S. B\"ohm]{S\"oren B\"ohm}
\address%[G.~R\"{o}hrle]
{Fakult\"at f\"ur Mathematik,
	Ruhr-Universit\"at Bochum,
	D-44780 Bochum, Germany}
\email{soeren.boehm@rub.de}

\author[B.\ Martin]{Benjamin Martin}
\address%[B.\ Martin]
{Department of Mathematics,
University of Aberdeen,
King's College,
Fraser Noble Building,
Aberdeen AB24 3UE,
United Kingdom}
\email{b.martin@abdn.ac.uk}
 
\author[G. R\"ohrle]{Gerhard R\"ohrle}
\address%[G.~R\"{o}hrle]
{Fakult\"at f\"ur Mathematik,
Ruhr-Universit\"at Bochum,
D-44780 Bochum, Germany}
\email{gerhard.roehrle@rub.de}

\begin{document}

\begin{abstract}
	Suppose $G$ is a simple algebraic group defined over an algebraically closed field of good characteristic $p$. 
	In 2018 Korhonen showed that if $H$ is a connected reductive subgroup of $G$ which contains a distinguished unipotent element $u$ of $G$ of order $p$, then $H$ is $G$-irreducible in the sense of Serre.  We present a short and uniform proof of this result under an extra hypothesis using so-called \emph{good $A_1$} subgroups of $G$, introduced by Seitz. In the process we prove some new results about good $A_1$ subgroups of $G$ and their properties.  We also formulate a counterpart of Korhonen's theorem for overgroups of $u$ which are finite groups of Lie type. Moreover, we generalize both results above by removing the restriction on the order of $u$ under a mild condition on $p$ depending on the rank of $G$, and we present an analogue of Korhonen's theorem for Lie algebras.
\end{abstract}

\maketitle

\section{Introduction and main results}
\label{sec:intro}

Throughout, $G$ is a connected reductive linear algebraic group
defined over an algebraically closed field $k$ 
of characteristic $p$ and 
$H$ is a closed subgroup of $G$.

Following  Serre \cite{serre2}, we say that 
$H$ is \emph{$G$-completely reducible} ($G$-cr for short) 
provided that whenever $H$ is contained in a parabolic subgroup $P$ of $G$,
it is contained in a Levi subgroup of $P$, 
and that $H$ is \emph{$G$-irreducible} ($G$-ir for short) provided $H$ is not contained in any proper parabolic subgroup of $G$ at all. Clearly, if $H$ is $G$-irreducible, it is trivially $G$-completely reducible, and an overgroup of a $G$-irreducible subgroup is again  $G$-irreducible; 
for an overview of this concept see \cite{BMR}, \cite{serre1} and \cite{serre2}.
Note that in case $G = \GL(V)$ a subgroup $H$ is $G$-completely reducible exactly when 
$V$ is a semisimple $H$-module and it is $G$-irreducible precisely when $V$ is an irreducible $H$-module.
Recall that if $H$ is $G$-completely reducible, 
then the identity component $H^\circ$ of $H$ is reductive,  \cite[Prop.~4.1]{serre2}.

A unipotent element $u$ of $G$ is  \emph{distinguished} provided any torus in the centraliser $C_G(u)$ of $u$ in $G$ is central in $G$.
Likewise, a nilpotent element $X$ of the Lie algebra $\frakg$ of $G$ is \emph{distinguished} provided any torus in the centraliser $C_G(X)$ of $X$ in $G$ is central in $G$, see \cite[\S 5.9]{carter:book} and \cite[\S 4.1]{Jantzen}.
For instance, regular unipotent elements in $G$ are distinguished, and so are regular nilpotent elements in $\frakg$ \cite[III 1.14]{springersteinberg} (or \cite[Prop.~5.1.5]{carter:book}).  The converse is true in type $A$, since a distinguished unipotent or nilpotent element must clearly consist of a single Jordan block.
Overgroups of regular unipotent elements have attracted much attention in the literature, e.g., see \cite{suprunenko}, \cite{saxl-seitz}, \cite{TZ},  \cite{MT}, and \cite{BMR:regular}.

In \cite{korhonen}, Korhonen proves the following remarkable result.
% in the special case when $G$ is simple.

\begin{thm}[{\cite[Thm.~6.5]{korhonen}}]
	\label{thm:korhonen}
	Suppose $G$ is simple and $p$ is good for $G$.
	Let $H$ be a reductive subgroup of $G$. Suppose $H^\circ$ contains a distinguished unipotent element of $G$ of order $p$. Then $H$ is $G$-irreducible.
\end{thm}

\noindent One can easily extend this theorem to arbitrary connected reductive $G$ by reducing to the simple case: see Remark~\ref{rem:simple_reduction}.

Korhonen's  proof  of  Theorem~\ref{thm:korhonen} depends on checks 
for the various possible Dynkin types for simple $G$. 
E.g., for $G$ simple of exceptional type, 
Korhonen's  argument relies on 
long exhaustive case-by-case investigations from \cite{LT},  where
all connected reductive non-$G$-cr subgroups are classified 
in the exceptional type groups in good characteristic.
For classical $G$, Korhonen requires an intricate  classification of 
all $\SL_2$-representations  on which a non-trivial unipotent element of $\SL_2$ acts with at most one Jordan
block of size $p$.
%As an application of Seitz's Theorem \ref{thm:seitz1}, 
Our main aim is to give a short uniform proof 
of Theorem \ref{thm:korhonen} 
in \S \ref{sec:proofs}
without resorting to further case-by-case checks, but imposing an extra hypothesis which allows us to
use a landmark result by Seitz (see \S \ref{sec:goodA1overgroups}).

\begin{thm}
\label{thm:dist-orderp}
 Suppose $p$ is good for $G$.  Let $H$ be a connected reductive subgroup of $G$.  Suppose $H$ contains a distinguished unipotent element of $G$ of order $p$.  Suppose also that
 
 \medskip
 \noindent $(\dagger)$ there exists a Springer map $\phi$ for $H$ such that $\phi(u)$ is a distinguished element of $\frakg$.

 \medskip 
 \noindent Then $H$ is $G$-irreducible.
\end{thm}

\noindent For a discussion of Springer maps, see Section~\ref{sec:springer}.

\begin{rem}
	\label{rem:classical}
	Suppose as in Theorem \ref{thm:korhonen}, that $G$ is simple classical with natural module $V$, and $p \ge \dim V > 2$. Then, thanks to \cite[Prop.\ 3.2]{jantzen0},  
	$V$ is semisimple as an $H^\circ$-module, and by \cite[(3.2.2(b))]{serre2}, this is equivalent to $H^\circ$ being $G$-cr. Then $H$ is $G$-ir, by 
	Lemma \ref{lem:GcrGir;dist}. This gives a short uniform proof of the conclusion of Theorem \ref{thm:korhonen} in this case, as the bound 
	$p \ge \dim V > 2$ ensures that every distinguished unipotent element (including the regular ones) is of order $p$. 
	The conclusion can fail if the bound is not satisfied: see Theorem \ref{thm:korhonen-bad}. 	
\end{rem}

We say that a subgroup of $G$ is of \emph{type $A_1$} if it is isomorphic to $\SL_2$ or $\PGL_2$.  Our proof of Theorem~\ref{thm:dist-orderp} involves the notion of a \emph{good $A_1$} subgroup, which was introduced by Seitz in \cite{seitz}.  We consider the interaction of good $A_1$ subgroups with associated cocharacters and Springer maps; we identify a useful class of Springer maps (Definition~\ref{defn:logarithmic}), which we call \emph{logarithmic} Springer maps, and we prove some results that are of interest in their own right (see Corollary~\ref{cor:compatible_good} and Lemma~\ref{lem:good_ascent}).  Our main result on good $A_1$ subgroups is the following (see Section~\ref{sec:good_characterisations} for definitions).

\begin{thm}
\label{thm:good_equivalences_intro}
 Suppose $p$ is good for $G$ and let $A$ be an $A_1$ subgroup of $G$.  The following  are equivalent.
 \begin{itemize}
  \item[(i)] $A$ is sub-principal.
  \item[(ii)] $A$ is optimal.
  \item[(iii)] $A$ is good.
 \end{itemize}
\end{thm}

Theorem~\ref{thm:korhonen} covers the situation when $p$ is good for $G$.  There are only a few cases when $G$ is simple, $p$ is bad for $G$, and $G$ admits a distinguished unipotent element of order $p$, by work of Proud-Saxl-Testerman \cite[Lem.~4.1, Lem.~4.2]{PST} (see Lemmas \ref{lem:pst41} and \ref{lem:pst42}).
In this case
the conclusion of Theorem \ref{thm:korhonen} fails precisely in one instance, as observed in \cite[Prop.~1.2]{korhonen} (Example \ref{ex:c2}), else it is valid (Example \ref{ex:g2}).
Combining the cases when $p$ is bad for $G$ with Theorem \ref{thm:dist-orderp}, we recover Korhonen's main theorem \cite[Thm.~1.3]{korhonen} (assuming that $(\dagger)$ from Theorem~\ref{thm:dist-orderp} holds).

\begin{thm}
	\label{thm:korhonen-bad}
	Suppose $G$ is simple and let $H$ be a reductive subgroup of $G$. 
	Suppose $H^\circ$ contains a distinguished unipotent element of $G$ of order $p$, and suppose that $(\dagger)$ holds. Then $H$ is $G$-irreducible, unless $p = 2$, $G$ is of type $C_2$, and $H$ is a type $A_1$ subgroup of $G$.
\end{thm}

Our next goal is an extension of Theorem \ref{thm:dist-orderp} to finite groups of Lie type in $G$.
Let $\sigma: G \to G$ be a Steinberg endomorphism of $G$,
so that the finite fixed point subgroup $G_\sigma = G(q)$ is a finite group of Lie type over the field $\BBF_q$ of $q$ elements. 
For a Steinberg endomorphism $\sigma$ of $G$ and a connected reductive $\sigma$-stable subgroup $H$ of $G$, $\sigma$ is also a Steinberg endomorphism for $H$ with finite fixed point subgroup  $H_\sigma = H \cap G_\sigma$, \cite[7.1(b)]{steinberg:end}.
Obviously, one cannot directly appeal to Theorem~\ref{thm:dist-orderp} to deduce anything about $H_\sigma$, because $(H_\sigma)^\circ$ is trivial.
For the notion of a $q$-Frobenius endomorphism, see \S \ref{sec:steinberg}.

\begin{thm}
	\label{thm:dist-orderp-finite}
	Let $H$ be a connected reductive subgroup of $G$ and  
	suppose $p$ is good for $G$. 
	Let $\sigma\colon G\to G$ be a Steinberg endomorphism stabilizing $H$ such that $\sigma|_H$ is a $q$-Frobenius endomorphism of $H$. If $G$ admits components of exceptional type, then assume $q > 7$.  Suppose $H_\sigma$ contains a distinguished unipotent element of $G$ of order $p$, and suppose that $(\dagger)$ holds.  Then $H_\sigma$ is $G$-irreducible. 
\end{thm}

Combining Theorem \ref{thm:dist-orderp-finite}  
with the aforementioned results from \cite{PST}, we are able to deduce the following 
analogue of Theorem \ref{thm:korhonen-bad} for finite subgroups of Lie type in $G$.

\begin{thm}
	\label{thm:korhonen-bad-finite}
	Let $H$ be a connected reductive subgroup of $G$. 
	Let $\sigma\colon G\to G$ be a Steinberg endomorphism stabilizing $H$ such that $\sigma|_H$ is a $q$-Frobenius endomorphism of $H$.  If $G$ is of exceptional type, then assume $q > 7$.  Suppose $H_\sigma$ contains a distinguished unipotent element of $G$ of order $p$, and suppose that $(\dagger)$ holds.  Then $H_\sigma$ is $G$-irreducible, unless $p = 2$, $G$ is of type $C_2$, and $H$ is a type $A_1$ subgroup of $G$.	
\end{thm}

In the special instance in Theorems \ref{thm:dist-orderp-finite} and \ref{thm:korhonen-bad-finite} when $H_\sigma$ contains a regular unipotent element $u$ from 
 $G$, the 
  conclusion of both theorems holds without 
  any restriction on the order of $u$ and without any restriction on $q$ (and without any exceptions of the type seen in Theorem \ref{thm:korhonen-bad-finite}); see \cite[Thm.~1.3]{BMR:regular}.
  
  In our final main result we show that we can remove condition $(\dagger)$ and the condition that $u$ has order $p$ from Theorem~\ref{thm:dist-orderp}, at the cost of increasing our bound on $p$.  We also obtain an analogue under the hypothesis that $\Lie(H)$ contains a distinguished nilpotent element of $\frakg$.  For a unipotent element $u\in G$ to be
distinguished is a mere condition on the structure of the centralizer $C_G(u)$ of $u$ in $G$. The extra condition 
for $u$ to have order $p$ is thus somewhat artificial. 
This restriction in Theorems~\ref{thm:korhonen} and \ref{thm:dist-orderp} is due to the methods used in \cite{korhonen} and in our proofs in \S \ref{sec:proofs}, which require the unipotent element to lie in a subgroup of type $A_1$; such an element must obviously have order $p$.

To state our theorem, we need to introduce an invariant $a(G)$ of $G$ from \cite[\S 5.2]{serre2}:
for $G$ simple, set $a(G) = \rk(G) +1$, 
where $\rk(G)$ is the rank of $G$.
For reductive $G$, let $a(G) = \max\{1, a(G_1), \ldots, a(G_r)\}$, 
where $G_1, \ldots, G_r$ are the simple components of $G$.

\begin{thm}
	\label{thm:dist}
	Suppose $p \ge a(G)$. Let $H$ be a reductive subgroup of $G$. 
	Suppose $H^\circ$ contains a distinguished unipotent element of $G$ or $\Lie(H)$ contains a distinguished nilpotent element of $\frakg$. 	Then  $H$ is $G$-irreducible.
\end{thm}

\smallskip

Section~\ref{sec:prelim} contains background material.  In Section~\ref{sec:arborder} we prove Theorem~\ref{thm:dist}, along with some analogues for finite subgroups of Lie type.  In Section~\ref{sec:springer_assoc} we discuss Springer maps and associated cocharacters.  We recall Seitz's notion of good $A_1$ subgroups in Section~\ref{sec:A1s} and we prove Theorem~\ref{thm:good_equivalences_intro} in Section~\ref{sec:good_characterisations} (see Theorem~\ref{thm:good_equivalences}).  Theorems \ref{thm:dist-orderp} and \ref{thm:korhonen-bad}--\ref{thm:korhonen-bad-finite} are proved in Section~\ref{sec:proofs}.

\section{Preliminaries}
\label{sec:prelim}

\subsection{Notation}
Throughout, we work over an algebraically closed field $k$ of characteristic $p$.   For convenience we assume that $p> 0$ unless otherwise stated; most of our results hold for $p= 0$ with obvious modifications and in many cases the proof is much easier (see Remark~\ref{rem:dist}(vi), for example). All affine varieties are considered over $k$ and are identified with their sets of $k$-points.
A linear algebraic group $H$ over $k$ has identity component $H^\circ$; if $H=H^\circ$, then we say that $H$ is \emph{connected}.
We denote by $R_u(H)$ the \emph{unipotent radical} of $H$; if $R_u(H)$ is trivial, then we say $H$ is \emph{reductive}.

Throughout, $G$ denotes a connected reductive linear algebraic group over $k$. All subgroups of $G$ considered are closed.
By $\DD G$ we denote the derived subgroup of $G$, and likewise for subgroups of $G$.
We denote the Lie algebra of $G$ by $\Lie (G)$ or by $\frakg$.  If $p> 0$ then we denote the $p$-power map on $\frakg$ by $X\mapsto X^{[p]}$.  By a Levi subgroup of $G$ we mean a Levi subgroup of some parabolic subgroup of $G$.  Recall that a homomorphism $f\colon G_1\to G_2$ of connected algebraic groups is a \emph{central isogeny} if $f$ is surjective, $\ker(f)$ is finite and the kernel of the derivative $df$ is central in $\Lie(G_1)$.

Let $Y(G) = \Hom(\BBG_m,G)$ denote the set of cocharacters of $G$. 
For $\mu\in Y(G)$ 
and $g\in G$ we define the \emph{conjugate cocharacter}
$g\cdot \mu \in Y(G)$ by 
$(g\cdot \mu)(t) = g\mu(t)g\inverse$ for $t \in \BBG_m$; 
this gives a left action of $G$ on $Y(G)$.
For $H$ a subgroup of $G$, let $Y(H): = Y(H^\circ) = \Hom(\BBG_m,H)$
denote the set of cocharacters of $H$. There is an obvious 
inclusion $Y(H) \subseteq Y(G)$.

Fix a Borel subgroup $B$ of $G$ containing a maximal torus $T$.
Let $\Phi = \Phi(G,T)$ be the root system of $G$ with respect to $T$,  
let $\Phi^+ = \Phi(B,T)$ be the set of positive roots of $G$, and let 
$\Sigma = \Sigma(G, T)$   
be the set of simple roots of $\Phi^+$. 
For each $\alpha\in \Phi$ we have a root subgroup $U_\alpha$ of $G$. For $\alpha$ in $\Phi$, let $x_\alpha : \BBG_a \to U_\alpha$ be a parametrization of the root subgroup $U_\alpha$ of $G$.

We denote the unipotent variety of $G$ by $\UU_G$ and the nilpotent cone of $\frakg$ by $\NN_G$.  We define
$$ \UU^{(1)}_G= \{u\in \UU_G \mid u^p= 1\} $$
and
$$ \NN^{(1)}_G= \{X\in \NN_G \mid  X^{[p]}= 0\}.$$

If $u\in \UU_G$ then we have a unique decomposition $u= u_1\cdots u_r$, where $u_i\in G_i$ and the $G_i$ are the simple factors of $\DD G$; we call $u_i$ the \emph{projection of $u$ onto $G_i$}.  Clearly $u$ is distinguished in $G$ if and only if $u_i$ is distinguished in $G_i$ for each $i$.

\subsection{Good primes}
\label{sec:goodprimes}
A prime $p$ is said to be \emph{good} for $G$
if it does not divide any coefficient of any positive root when expressed as a linear combination of simple ones. Else $p$ is called \emph{bad} for $G$, \cite[\S 4]{springersteinberg}.
Explicitly, if $G$ is simple, $p$ is \emph{good} for $G$ provided $p > 2$ in case $G$ is of Dynkin type $B_n$, $C_n$, or $D_n$; $p > 3$ in case $G$ is of Dynkin type $E_6$, $E_7$, $F_4$ or $G_2$ and $p > 5$ in case $G$ is of type $E_8$.  If $G$ is semisimple then we say that $p$ is \emph{separably good} for $G$ if $p$ is good for $G$ and the canonical map from $G_{\rm sc}$ to $G$ is separable, where $G_{\rm sc}$ is the simply connected cover of $G$.  For arbitrary connected reductive $G$ we say that $p$ is \emph{separably good} for $G$ if it is separably good for $[G, G]$.  We observe that if $L$ is a Levi subgroup of $G$ and $p$ is good for $G$, then it is also good for $L$.

\subsection{Steinberg endomorphisms of $G$}
\label{sec:steinberg}

Recall that a \emph{Steinberg endomorphism} of $G$ is a surjective homomorphism $\sigma:G\to G$ such that the 
corresponding fixed point subgroup $G_\sigma :=\{g \in G \mid \sigma(g) = g\}$ of $G$ is finite.
Frobenius endomorphisms $\sigma_q$ of reductive groups over finite fields are familiar examples, giving rise to \emph{finite groups of Lie type} $G(q)$.  See 
Steinberg \cite{steinberg:end} for a detailed discussion
(for this terminology, see \cite[Def.\ 1.15.1b]{GLS:1998}). 
The set of all Steinberg endomorphisms of $G$ is a subset of the set of all 
isogenies $G\rightarrow G$ (see \cite[7.1(a)]{steinberg:end}) 
which encompasses in particular all generalized Frobenius endomorphisms, 
i.e., endomorphisms of $G$ some power of which are Frobenius endomorphisms 
corresponding to some $\BBF_q$-rational structure on $G$. 
In that case we also denote the finite group of Lie type $G_\sigma$ by $G(q)$.
If $\SS$ is a $\sigma$-stable set of closed subgroups of $G$, then $\SS_\sigma$ denotes the subset consisting of all  $\sigma$-stable members of $\SS$.

If $\sigma_q:G\to G$ is a standard $q$-power Frobenius endomorphism of $G$, then there exist a $\sigma_q$-stable
maximal torus $T$ and Borel subgroup $B\supseteq T$, and with respect to a chosen parametrisation of
the root groups as above, we have $\sigma_q(x_\alpha(t)) = x_\alpha(t^q)$ for each $\alpha\in \Phi$ and $t\in \BBG_a$, see \cite[Thm.~1.15.4(a)]{GLS:1998}.
Following \cite{PST}, we call a generalized Frobenius endomorphism $\sigma$
a \emph{$q$-Frobenius endomorphism} provided $\sigma = \tau \sigma_q$, where $\tau$ is an algebraic automorphism of $G$ of finite order, $\sigma_q$ is a standard $q$-power Frobenius endomorphism of $G$, and $\sigma_q$ and $\tau$ commute.  When $p$ is bad for $G$, a $q$-Frobenius endomorphism
does not involve a twisted Steinberg endomorphism, see \cite[\S 3]{PST}. 
If $G$ is simple and $p$ is good for $G$, then any Steinberg endomorphism of $G$ is a $q$-Frobenius endomorphism, 
see \cite[\S 11]{steinberg:end}.
If $G$ is not simple and $p$ is bad for $G$, then a generalized Frobenius map may fail to factor into
a field and algebraic automorphism of $G$, e.g., see \cite[Ex.~1.3]{HRG}.

\subsection{Bala-Carter Theory}
\label{sub:balacarter}

We recall some relevant results and concepts from Bala-Carter theory. Suppose $p$ is good for $G$.
A parabolic subgroup $P$ of $G$ admits a dense open orbit on its unipotent radical $R_u(P)$, the so-called \emph{Richardson orbit}; see \cite[Thm.~5.2.1]{carter:book}.
A parabolic subgroup $P$ of $G$ is called \emph{distinguished} provided $\dim (\DD P/R_u(P)) = \dim (R_u(P)/\DD R_u(P))$, see \cite[\S 2.1]{premet}.
For $G$ simple, the distinguished parabolic subgroups of $G$ (up to $G$-conjugacy) were 
worked out in \cite{BaCaI} and \cite{BaCaII}; see 
\cite[pp.\ 174--177]{carter:book}.
The notion of a distinguished parabolic subgroup of $G$ also makes  sense in case $p$ is bad for $G$, cf.~\cite[\S 4.10]{Jantzen}. 

The following is the celebrated Bala-Carter Theorem, see \cite[Thm.~5.9.5, Thm.~5.9.6]{carter:book}, which is valid in good characteristic, thanks to work of Pommerening \cite{pommereningI}, \cite{pommereningII}. For the Lie algebra versions see also 
\cite[Prop.~4.7, Thm.~4.13]{Jantzen}. 

\begin{thm}\label{thm:bala-carter}
	Suppose $p$ is good for $G$. 
	\begin{itemize}
		\item [(i)] 	There is a bijective map between the $G$-conjugacy classes of distinguished unipotent elements of $G$ and conjugacy classes of distinguished parabolic subgroups of $G$. The unipotent class corresponding to a given parabolic subgroup $P$ contains the dense $P$-orbit on $R_u(P)$. 
		\item [(ii)] 	There is a bijective map between the $G$-conjugacy classes of unipotent elements of $G$ and conjugacy classes of pairs $(L, P)$, where $L$ is a Levi subgroup of $G$ and $P$ is a 
		distinguished parabolic subgroup of $\DD L$. The unipotent class corresponding to the pair $(L, P)$ contains the dense $P$-orbit on $R_u(P)$. 
	\end{itemize}
\end{thm}

\begin{rem}
\label{rem:sigmastablelevi}
(i). Let $1 \ne u \in \UU_G$. 
Let $S$ be a maximal torus of $C_G(u)$. Then $u$ is distinguished 
in the Levi subgroup $C_G(S)$ of $G$, since $S$ is the unique maximal torus of 
$C_{C_G(S)}(u)$. 
Conversely, 
if $L$ is a Levi subgroup of $G$ with $u$ distinguished in $L$,
then the connected center of $L$ is a maximal torus of $C_G(u)^\circ$, see \cite[Rem.\ 4.7]{Jantzen}.

(ii). Let $\sigma : G \to G$ be a Steinberg endomorphism of $G$ and let $1 \ne u \in G_\sigma$ be unipotent. Then $C_G(u)^\circ$ is $\sigma$-stable. The set of all maximal tori of $C_G(u)^\circ$ is 
$\sigma$-stable and $C_G(u)^\circ$ is transitive on that set, \cite[Thm.~6.4.1]{springer}. Thus the Lang-Steinberg Theorem, see \cite[I 2.7]{springersteinberg}, provides a $\sigma$-stable maximal torus, say $S$, of $C_G(u)^\circ$. Then, by part (i), $L = C_G(S)$ is a $\sigma$-stable Levi subgroup of $G$ and $u$ is distinguished in $L$.  
\end{rem}

\subsection{Cocharacters and parabolic subgroups of $G$}
\label{sub:parabolics}

Let $\lambda \in Y(G)$. 
Recall that $\lambda$ affords a 
$\ZZ$-grading on $\frakg = \bigoplus_{j \in \ZZ}\frakg(j, \lambda)$, where $\frakg(j, \lambda) := \{X \in \frakg \mid \Ad(\lambda(t))X = t^jX 
\text{ for every } t \in \BBG_m\}$ 
is the $j$-weight space of $\Ad(\lambda(\BBG_m))$ on $\frakg$, 
see 
\cite[\S 5.5]{carter:book} or \cite[\S 5.1]{Jantzen}.
Let $\frakp_\lambda: = \bigoplus_{j \ge 0}\frakg(j, \lambda)$.
Then there is a unique parabolic subgroup $P_\lambda$ with 
$\Lie(P_\lambda) = \frakp_\lambda$ and $C_G(\lambda) := C_G(\lambda(\BBG_m))$ is a Levi subgroup of $P_\lambda$.
Since all maximal tori in $G$ are conjugate, 
it suffices to describe these subgroups and subalgebras when $\lambda \in Y(T)$ for our fixed maximal torus $T$.
In this case, letting $X(T) = \Hom(T,\BBG_m)$ denote the character group of $T$,  we have $U_\alpha\subseteq P_\lambda$ 
if and only if $\langle\lambda,\alpha\rangle \geq 0$, where $\langle \, ,\, \rangle:Y(T)\times X(T)\to \ZZ$ is the usual pairing between cocharacters and characters.
We have $U_\alpha\subseteq C_G(\lambda)$ if and only if $\langle\lambda,\alpha\rangle = 0$,
and $R_u(P_\lambda)$ is generated by the $U_\alpha$ with $\langle\lambda,\alpha\rangle >0$; see the proof of \cite[Prop.~8.4.5]{springer}.

Set $J:=\{\alpha\in\Sigma \mid\langle\alpha,\lambda\rangle=0\}$.
Then $P_\lambda=P_J=\big\langle T,U_\alpha \mid \langle\alpha,\lambda\rangle\geq0 \big\rangle$ is the \emph{standard parabolic subgroup} of $G$ associated with  $J \subseteq \Sigma$.

Let $\rho = \sum_{\alpha \in \Sigma} c_{\alpha\rho} \alpha$
be the highest root in $\Phi^+$. Define $\hgt_J(\rho):= \sum_{\alpha \in \Sigma\setminus J} c_{\alpha\rho}$.
In view of Theorem \ref{thm:bala-carter}, 
the following gives the order of a distinguished unipotent element in good characteristic.

\begin{lem}[{\cite[Order Formula 0.4]{testerman}}]
	\label{lem:order}	
	Suppose $p$ is good for $G$. Let $P = P_J$ be a distinguished parabolic subgroup of $G$ and let $u$ be in the Richardson orbit of $P$ on $R_u(P)$. Then the order of $u$ is $\min\{p^a \mid p^a > \hgt_J(\rho)\}$.
\end{lem}

\subsection{Overgroups of type $A_1$}
\label{subsec:A1s}
It has been understood for some time now that if $p$ is good for $G$ then one can study unipotent elements of $G$ having order $p$ by embedding them in $A_1$ subgroups of $G$.  The existence of $A_1$ overgroups for unipotent elements of order $p$ is guaranteed by the following fundamental results of Testerman \cite[Thm.~0.1]{testerman} if $p$ is good for $G$ and else by Proud-Saxl-Testerman \cite{PST}; these results were originally proved for semisimple $G$ but the extension to arbitrary connected reductive $G$ is immediate. 

\begin{thm}[{\cite[Thm.~0.1, Thm.~0.2]{testerman}}]
	\label{thm:A1pgood}
	Suppose $p$ is good for $G$.
	Let $\sigma$ be $\id_G$ or a Steinberg endomorphism of $G$. Let $u \in G_\sigma$ be unipotent of order $p$. Then there exists a $\sigma$-stable subgroup of $G$ of type $A_1$ containing $u$.
\end{thm}

The proof of Theorem \ref{thm:A1pgood} is based on case-by-case checks and depends in part on computer calculations involving  explicit unipotent class representatives. For a uniform proof of the theorem, we refer the reader to McNinch \cite{mcninch1}.  Conditions to ensure $G$-complete reducibility of such a subgroup were given in \cite{McNT}.

We now consider $A_1$ overgroups of distinguished unipotent elements in arbitrary characteristic.  There are only a few instances when $G$ is simple, $p$ is bad for $G$, and $G$ admits a distinguished unipotent element of order $p$.  We recall the relevant results concerning the existence of $A_1$ overgroups of such elements from \cite{PST}.

\begin{lem}[{\cite[Lem.~4.1]{PST}}]
	\label{lem:pst41}
	Let $G$ be simple classical of type $B_l, C_l$, or $D_l$ and suppose $p = 2$.
	Then $G$ admits a distinguished involution $u$ if and only if $G$ is of type $C_2$ and $u$ belongs to the subregular class $\mathscr C$ of $G$.
	If $\sigma$ is $\id_G$ or a $q$-Frobenius  endomorphism of $G$ and $u \in \mathscr C \cap G_\sigma$, then there exists a $\sigma$-stable subgroup $A$ of $G$ of type $A_1$ containing $u$.
\end{lem}

\begin{ex}
	\label{ex:c2}
	Let $G$ be simple of 
	type $C_2$ and let $p = 2$.  Let $\sigma$ be $\id_G$ or a $q$-Frobenius  endomorphism of $G$, and suppose $u\in G_\sigma$ is a distinguished unipotent element of order $2$.
  Then Lemma \ref{lem:pst41} provides a $\sigma$-stable subgroup $A$ of type $A_1$ containing $u$. Thanks to \cite[Prop.~1.2]{korhonen}, there are such subgroups $A$ which are not $G$-ir.
	In fact, according to \emph{loc.~cit.},  
	there are two $G$-conjugacy classes of such $A_1$ subgroups in $G$; 
see Example \ref{ex:goodA1} below.
Since $A$ is contained in a proper parabolic subgroup of $G$, so is $A_\sigma$. So the latter is also not $G$-ir.  By Lemma~\ref{lem:GcrGir;dist} below, $A$ and $A_\sigma$ are not $G$-cr, either.
\end{ex}

\begin{lem}[{\cite[Lem.~3.3, Lem.~4.2]{PST}}]
	\label{lem:pst42}
	Let $G$ be simple of exceptional type and suppose $p$ is bad for $G$.
	Then $G$ admits a distinguished unipotent element $u$ of order $p$ if and only if $G$ is of type $G_2$, $p = 3$,  and $u$ belongs either to the subregular class $G_2(a_1)$\footnote{Throughout, we use the Bala-Carter notation for distinguished classes in the exceptional groups, see \cite[\S 5.9]{carter:book}.}  or to the class $A_1^{(3)}$ of $G$.
	Moreover, if $\sigma$ is $\id_G$ or a $q$-Frobenius  endomorphism of $G$ and $u \in G_2(a_1) \cap G_\sigma$, then there exists a $\sigma$-stable subgroup $A$ of $G$ of type $A_1$ containing $u$.
	In case $u \in A_1^{(3)}$, there is no overgroup of $u$ in $G$ of type $A_1$.
\end{lem}

\begin{ex}
	\label{ex:g2}
	Let $G$ be simple of type $G_2$ and $p = 3$. Let $H$ be a reductive subgroup of $G$
	containing a distinguished unipotent element $u$ from $G$. 
	Then, as $p = 3 = a(G_2)$, it follows from Theorem \ref{thm:dist} that
	$H^\circ$ is $G$-ir, and so is $H$. 	This applies in particular to 
	the  subgroup  $A$  of $G$ of type $A_1$ 
	containing $u$ when $u\in G_2(a_1)$.
	Since $3$ is not a good prime for $G$, Theorem \ref{thm:korhonen} does not apply in this case.
			See also \cite[Cor.~2]{stewart:g2}.
			
	In case of the presence of a $q$-Frobenius endomorphism of $G$ stabilizing $H$, 
	we show in our proof of Theorem \ref{thm:korhonen-bad-finite} that
	$H_\sigma$ is also $G$-ir.
\end{ex}

\begin{thm}[{\cite[Thm.~5.1]{PST}}]
	\label{thm:A1pbad}
	Let $G$ be semisimple and suppose $p$ is bad for $G$. 
 Let $\sigma$ be $\id_G$ or a $q$-Frobenius  endomorphism of $G$. Let $u \in G_\sigma$ be unipotent of order $p$.
	If $p = 3$, and $G$ has a simple component of type $G_2$, assume that the projection of $u$ into this component does not lie in the class $A_1^{(3)}$.  Then there exists a $\sigma$-stable subgroup of $G$ of type $A_1$ containing~$u$.
\end{thm}

\begin{cor}
	\label{cor:g2}
	Let $G$ be simple of type $G_2$, $p=3$ and 
	let $\sigma$ be $\id_G$ or a $q$-Frobenius  endomorphism of $G$. Let $u \in A_1^{(3)} \cap G_\sigma$. Then there is no proper semisimple subgroup $H$ of $G$ containing $u$. In particular, 
	any such $u$ is \emph{semiregular}, that is, $C_G(u)$ does not contain a non-central semisimple element of $G$.
\end{cor}

\begin{proof}
	By way of contradiction, suppose $H$ is a proper semisimple subgroup of $G$ containing $u$.
	Since $p = 3$ is good for $H$ (e.g., see \cite[Cor.~3]{stewart:g2}), there is a 
	$\sigma$-stable $A_1$ subgroup $A$ in $H$ containing $u$, by Theorem \ref{thm:A1pgood}.
 It follows from Lemma \ref{lem:pst42} that $u\in G_2(a_1)$ which contradicts the hypothesis that $u\in A_1^{(3)}$.
\end{proof}

The following result is needed in the proof of Theorem~\ref{thm:dist-orderp} below.

\begin{lem}
	\label{lem:g2p=3}
	Suppose $G$ is semisimple and $p=3$ is good for $G$. 
	Let $H$ be a connected reductive subgroup of $G$. Let $u \in H$ be a unipotent element of order $3$ which is distinguished in $G$.
	Then $H$ does not admit a simple component of type $G_2$.
\end{lem}

\begin{proof}(cf.~\cite[p.~387]{korhonen})
	Since $p$ is good for $G$, every simple component of $G$ is
	of classical type. Let $V'$ be the natural module of the simple component $G'$ of $G$, and let $H'$ be the projection of $H$ into $G'$. Since the projection $u'$ of $u$ into $G'$ has order $3$, the largest Jordan block size of $u'$ on $V'$ is at most $3$. Since $u'$ is distinguished in $G'$, the Jordan block sizes of $u'$ are distinct and of the same parity. Hence $\dim V' \le 4$. Since a non-trivial representation of a simple algebraic group of type $G_2$ has dimension at least $5$, $H'$ does not have a simple component of type $G_2$.  Hence $H$ has no simple component of type $G_2$.
\end{proof}

In summary, we see that if $1\neq u\in \UU_G^{(1)}$ then $u$ is contained in an $A_1$ subgroup of $G$ unless $p=3$ and $G$ has a simple $G_2$ factor such that the projection of $u$ onto this factor lies in the class $A_1^{(3)}$.

\section{Variations on Theorems \ref{thm:dist-orderp} and \ref{thm:dist-orderp-finite}}
%\section{Proof of Theorem~\ref{thm:dist}}
\label{sec:arborder}

In this section we prove Theorem~\ref{thm:dist}.  We also state and prove some related results for finite subgroups of Lie type.  We need the following analogue of \cite[Cor.~4.6]{BMR:regular}, which
shows that in order to derive 
the $G$-irreducibility of $H$ in 
Theorem \ref{thm:dist},
it suffices to show that $H$ is $G$-cr; see also \cite[Lem.~6.1]{korhonen}. This also applies to Theorem \ref{thm:dist-orderp} and 
Theorem \ref{thm:dist-orderp-finite}.

\begin{lem}
	\label{lem:GcrGir;dist}
	Let $H$ be a $G$-completely reducible subgroup of $G$. Suppose that  $H$ contains a distinguished unipotent element $u$ of $G$ or $\Lie(H)$ contains a distinguished nilpotent element $X$ of $\frakg$.  Then  $H$ is $G$-irreducible.
\end{lem}

\begin{proof}
	Suppose $H$ is contained in a parabolic subgroup $P$ of $G$.  Then,  
	by hypothesis, $H$ is contained in a Levi subgroup $L$ of $P$.  
	As the latter is the centraliser of a torus $S$ in $G$, 
	$S$ centralises $u$ (resp.,~$X$) and so $S$ is central in $G$. Hence $L = G$, which implies $P= G$.
\end{proof}

Along with Lemma~\ref{lem:GcrGir;dist}, the following theorem of Serre immediately yields Theorem \ref{thm:dist}.

\begin{thm}[{\cite[Thm.\ 4.4]{serre2}}]
	\label{thm:ag}
	Suppose $p \ge a(G)$ and  
	$(H:H^\circ)$ is prime to $p$.
	Then $H^\circ$ is reductive  if and only if 
	$H$ is $G$-completely reducible.
\end{thm}

\begin{proof}[Proof of Theorem \ref{thm:dist}]
	Since $p \ge a(G)$, Theorem \ref{thm:ag} applied to $H^\circ$ shows the latter is $G$-cr.
	Thus $H^\circ$ is $G$-ir by Lemma \ref{lem:GcrGir;dist}, and so is $H$.
\end{proof}

\begin{rems}
	\label{rem:dist}
	(i). The characteristic restriction in Theorem \ref{thm:dist} (and Theorem \ref{thm:ag}) is needed; see Theorem \ref{thm:korhonen-bad}. 
	
	(ii). The condition  in Theorem \ref{thm:dist} that the distinguished unipotent element of $G$ belongs to $H^\circ$ (as opposed to $H$) is also necessary, as for instance the finite unipotent subgroup of $G$ generated by a given distinguished unipotent element of $G$ is not $G$-cr \cite[Prop.~4.1]{serre2}. 
	
	(iii). Under the given hypotheses, Theorem~\ref{thm:dist} applies to an arbitrary distinguished unipotent element of $G$, irrespective of its order. For Theorem \ref{thm:korhonen} to achieve the same uniform result, $p$ has to be sufficiently large to guarantee that the chosen element has order $p$.
	For $G$ simple classical with natural module $V$, this requires the bound $p \ge \dim V$; see Remark \ref{rem:classical}.  
	For $G$ simple of exceptional type, this requires 
	the following bounds: $p > 11$ for $E_6$, 
	$p > 17$ for $E_7$, $p > 29$ for $E_8$, $p > 11$ for $F_4$, and $p > 5$ for $G_2$; see \cite[Prop.~2.2]{testerman}.  So in many cases the bound $p\geq a(G)$ from Theorem~\ref{thm:dist} is better.
	
	(iv).
	For an instance when $p$ is bad for $G$ 
	so that Theorem \ref{thm:korhonen} does not apply, 
	but Theorem \ref{thm:dist} does, see Example \ref{ex:g2}.
	
	(v). Theorem \ref{thm:dist} generalizes
	\cite[Thm.~3.2]{BMR:regular} which consists of 
	the analogue in the special instance when the distinguished element is regular in $G$ (or $\frakg$). Note that in this case no restriction on $p$ is needed, see \cite[Thm.~1.2]{TZ}, \cite[Thm.~1]{MT}, \cite[Thm.~3.2]{BMR:regular}.
	
	(vi). In characteristic $0$, a subgroup $H$ of $G$ is $G$-cr if and only if it is reductive, \cite[Prop.~4.1]{serre2}. So in that case the conclusion of Theorem \ref{thm:dist} follows  directly from Lemma \ref{lem:GcrGir;dist}.	  
\end{rems}

Once again, in the presence of a Steinberg endomorphism $\sigma$ of $G$, one cannot  appeal to Theorem~\ref{thm:dist} directly to deduce anything about $H_\sigma$, because $(H_\sigma)^\circ$ is trivial. In Corollary \ref{cor:dist-finite2} we present an analogue of Theorem \ref{thm:dist} for the finite groups of Lie type $H_\sigma$ under an additional condition stemming from \cite{BBLMR}.

Note that for $S$ a torus in $G$, we have $C_G(S) = C_G(s)$ for some  $s \in S$, see \cite[III Prop.~8.18]{borel}.

\begin{prop} [{\cite[Prop.~3.2]{BBLMR}}]
	\label{prop:HcrvsHsigmacr}
	Let $H \subseteq G$ be connected reductive groups. 
	Let $\sigma\colon G\to G$ be a Steinberg endomorphism that stabilises $H$  and a maximal  torus $T$  of $H$. Suppose
	\begin{itemize}
		\item [(i)]  $C_G(T) = C_G(t)$, for some $t \in T_\sigma$,   and
		\item [(ii)] $H_\sigma$ meets every $T$-root subgroup of $H$ non-trivially.  
	\end{itemize}
	Then $H_\sigma$ and $H$ belong to the same 
	parabolic and the same Levi subgroups of $G$.
	In particular, $H$ is $G$-completely reducible if and only if $H_\sigma$ is $G$-completely reducible; similarly, $H$ is $G$-irreducible if and only if $H_\sigma$ is $G$-irreducible.
\end{prop}

Without condition (i), the proposition is false in general, 
see \cite[Ex.~3.2]{BBLMR}. 
The following is 
an immediate consequence of 
Theorem \ref{thm:dist} and 
Proposition \ref{prop:HcrvsHsigmacr}.

\begin{cor}
	\label{cor:dist-finite2}	
	Suppose $G, H$ and $\sigma$ satisfy the hypotheses of Proposition \ref{prop:HcrvsHsigmacr}. Suppose in addition that $p\geq a(G)$. 
	If $H_\sigma$ contains a distinguished unipotent element of $G$, then 
	$H_\sigma$ is $G$-irreducible. 
\end{cor}

Corollary \ref{cor:dist-finite2} generalizes
\cite[Thm.~1.3]{BMR:regular} which consists of the analogue in the special instance when the distinguished element is regular in $G$. Note in this case no restriction on $p$ is needed.

The following example shows that the conditions 
in 
Corollary \ref{cor:dist-finite2}
hold generically. 

\begin{ex}
	\label{ex:Tsigma}
	Let $\sigma_q \colon \GL(V) \to \GL(V)$ be a standard Frobenius endomorphism which stabilises a connected reductive subgroup $H$ of $\GL(V)$  and a maximal  torus $T$  of $H$.  
	Pick $l \in \BBN$ such that firstly all the different $T$-weights of $V$ are still distinct when restricted to $T_{\sigma_q^l}$ and secondly there is a $t \in T_{\sigma_q^l}$, such that $C_{\GL(V)}(T) = C_{\GL(V)}(t)$.
	Then for every $n \ge l$, both conditions in Corollary \ref{cor:dist-finite2} are satisfied for $\sigma = \sigma_q^n$. Thus there are only finitely many powers of $\sigma_q$ for which the conditions 
	in 
	Corollary \ref{cor:dist-finite2} can fail.
	The argument here readily generalises to a Steinberg endomorphism of  a connected reductive $G$ which induces a generalised Frobenius morphism on $H$.
\end{ex}

\section{Springer maps and associated cocharacters}
\label{sec:springer_assoc}

\subsection{Springer maps}
\label{sec:springer}
The notion of a Springer isomorphism was introduced in
\cite{springer0}.  A \emph{Springer isomorphism} is a $G$-equivariant isomorphism of varieties $\phi\colon \UU_G \to \NN_G$.  It follows from work of Springer \cite[Thm.\ 3.1]{springer0} that a Springer isomorphism $\phi$ exists if $p$ is good and $G$ is simple and simply connected.  We follow Springer and consider $G$-equivariant maps from $\UU_G$ to $\NN_G$, but note that several other authors consider $G$-equivariant maps from $\NN_G$ to $\UU_G$ instead (see, e.g., \cite{sobaje}).

We wish to consider versions of Springer maps for arbitrary connected reductive $G$.  To prove existence, we need to weaken the definition slightly.

\begin{defn}
\label{defn:Springer_map}
 A \emph{Springer map (for $G$)} is a $G$-equivariant homeomorphism of varieties $\phi\colon \UU_G \to \NN_G$.
\end{defn}

\begin{rem}
\label{rem:distinguished_pres}
 It follows from $G$-equivariance that if $\phi$ is a Springer map then $\phi(1)= 0$ and for any $u\in \UU_G$, $u$ is distinguished if and only if $\phi(u)$ is distinguished.
\end{rem}

\begin{rem}
\label{rem:springer_homeom}
 If $p$ is good for $G$ then there exists a Springer map $\phi$ for $G$: see \cite[Prop.~5]{mcninchsommers}.  Below we sketch the argument briefly, following {\em loc.\ cit.} and \cite[\S  1.2]{sobaje}.  Note first that a Springer map is uniquely determined by its value on a single regular unipotent element $u$ of $G$: this follows from $G$-equivariance, and because the orbit $G\cdot u$ is dense in $\UU_G$.  If $G$ is simple and $p$ is separably good for $G$ then we can prove existence of a Springer isomorphism by reversing this argument.  Fix a regular unipotent element $u\in G$, and choose $X\in \NN_G$ such that $C_G(u)= C_G(X)$.  We have an obvious isomorphism from $G\cdot u$ to $G\cdot X$.  Because $\UU_G$ and $\NN_G$ are normal (for references, see \cite[Lecture 2]{serre1}), one can show that this map extends to a unique $G$-equivariant isomorphism from $\UU_G$ to $\NN_G$.  Let us say that $G$ is \emph{of separable type} if it is of the form $G= G_1\times\cdots \times G_r$, where each $G_i$ is simple and $p$ is separably good for $G$.  A similar argument to the above works for $G$ of separable type: for $\UU_G= \UU_{G_1}\times\cdots\times \UU_{G_r}$ is normal since each $\UU_{G_i}$ is, and likewise $\NN_G$ is normal.

Now let $G$ be an arbitrary connected reductive group and assume $p$ is good for $G$.  Since $\UU_G\subseteq \DD G$ and $\NN_G\subseteq \Lie(\DD G)$, there is no harm in assuming that $G$ is semisimple.  Choose a central isogeny $\pi$ from $\widetilde{G}$ to $G$, where $\widetilde{G}= \widetilde{G}_1\times\cdots\times \widetilde{G}_r$ with each $\widetilde{G}_i$ simple and $p$ separably good for $\widetilde{G}$.  Then $\pi$ (resp., $d\pi$) gives a homeomorphism from $\UU_{\widetilde{G}}$ to $\UU_G$ (resp., from $\NN_{\widetilde{G}}$ to $\NN_G$) \cite[Lem.\ 27]{mcninch1}.  If $\widetilde{\phi}$ is a Springer map for $\widetilde{G}$ then the composition $\UU_G\to \UU_{\widetilde{G}}\stackrel{\widetilde{\phi}}{\to} \NN_{\widetilde{G}}\to \NN_G$ is a Springer map for $G$.  This gives a bijection between the set of Springer maps for $\widetilde{G}$ and the set of Springer maps for $G$.  Since $\widetilde{G}$ admits a Springer isomorphism, it follows that $G$ admits a Springer map.

Note that if $G$ is of separable type then any Springer map $\phi$ for $G$ is an isomorphism.  For fix a regular unipotent element $u\in G$ and let $X= \phi(u)$.  By the above discussion, there is a unique Springer isomorphism $\phi'$ for $G$ such that $\phi'(u)= X$; the uniqueness implies that $\phi'= \phi$.  It also follows from the construction in the previous paragraph that if $G$ is an arbitrary connected reductive group and $p$ is good for $G$ then the restriction of $\phi$ to any maximal unipotent subgroup $U$ of $G$ gives an isomorphism of varieties from $U$ to $\Lie(U)$.
\end{rem}

\begin{rem}
\label{rem:Springer_product}
 Let $G_1, G_2$ be connected reductive groups and let $\phi_i$ be a Springer map for $G_i$ for $i= 1,2$.  We claim that the map $\phi_1\times \phi_2\colon \UU_{G_1\times G_2}\to \NN_{G_1\times G_2}$ given by $(\phi_1\times \phi_2)((u_1,u_2))= (\phi_1(u_1), \phi_2(u_2))$ is a Springer map for $G_1\times G_2$.  It is clear that $\phi_1\times \phi_2$ is a $(G_1\times G_2)$-equivariant bijection.  The Zariski topology on the product of varieties is not the product topology, so it is not immediately clear that $\phi_1\times \phi_2$ is a homeomorphism.  To see this, we can pass to the case when $G_1$ and $G_2$ are of separable type, by Remark~\ref{rem:springer_homeom}.  Then $\phi_1$ and $\phi_2$ are isomorphisms, so $\phi_1\times \phi_2$ is an isomorphism, and the claim follows.  We show in Lemma~\ref{lem:all_product_Springer} that every Springer map for $G_1\times G_2$ arises in this way.
\end{rem}

\begin{rem}
\label{rem:conj_class_map}
 It follows from $G$-equivariance that a Springer map $\phi$ gives rise to a bijective map from the set of unipotent conjugacy classes of $G$ to the set of nilpotent conjugacy classes of $\frakg$.  Serre shows \cite[\S  10, Corollary]{mcninch2} that this map does not depend on the choice of Springer map (the proof given in {\em loc.\ cit.}\ is for simple $G$, but the extension to arbitrary $G$ follows easily from Remarks~\ref{rem:springer_homeom} and \ref{rem:Springer_product}).  In particular, the condition in $(\dagger)$ does not depend on the choice of Springer map for $H$.
\end{rem}

\begin{rem}
 Springer maps need not exist in bad characteristic.  For instance, a simple group $G$ of type $F_4$ with $p= 2$ does not admit a Springer map, because the numbers of unipotent classes in $G$ and nilpotent $G$-orbits in $\Lie(G)$ are different (see \cite[\S  5.11]{carter:book}).
\end{rem}

The following result is \cite[\S1.2, Rem.\ 1]{sobaje}.

\begin{lem}
\label{lem:p_preserve}
 Let $\phi$ be a Springer map for $G$.  Then for any $u\in \UU_G$, $\phi(u^p)= \phi(u)^{[p]}$.
\end{lem}

\begin{rem}
\label{rem:1-preserving}
 It follows from Lemma~\ref{lem:p_preserve} that any Springer map for $G$ induces a homeomorphism from $\UU^{(1)}_G$ to $\NN^{(1)}_G$.
\end{rem}

In Section~\ref{sub:cochars} we define the notion of an associated cocharacter for an element $u\in \UU_G$, using a fixed Springer map to give a correspondence between $\UU_G$ and $\NN_G$.  In many contexts one can fix a single Springer map once and for all.  We need, however, to consider the interaction of Springer maps with subgroups of $G$.  This motivates the following definition.

\begin{defn}
 Let $M$ be a connected subgroup of $G$.  We say that a Springer map $\phi$ for $G$ is \emph{$M$-compatible} if $\phi(\UU_M)\subseteq \NN_M$, and we say that $M$ is \emph{Springer-compatible} if there exists an $M$-compatible Springer map for $G$.
\end{defn}

If $\phi$ is $M$-compatible then in fact $\phi(\UU_M)= \NN_M$, since $\dim(\UU_M)= \dim(\NN_M)$; note that dimension can be defined in a purely topological way (via Krull dimension), so it is preserved by homeomorphisms.  Note also that when $M$ is reductive and $\phi$ is an $M$-compatible Springer map, the restriction of $\phi$ to $\UU_M$ gives a Springer map for $M$, which we denote by $\phi_M$.

\begin{ex}
\label{ex:Springer_cent}
 (\cite[(3.3.1)(a)]{McNT2})  Let $M$ be a connected reductive subgroup of the form $C_G(S)^\circ$, where $S\subseteq G$.  It follows from $G$-equivariance that any Springer map for $G$ is $M$-compatible, so $M$ is Springer-compatible.
\end{ex}

\begin{ex}
\label{ex:Springer_factor}
 The arguments in Remark~\ref{rem:springer_homeom} show that if $G_i$ is a simple factor of $G$ then any Springer map for $G$ is $G_i$-compatible, so $G_i$ is Springer-compatible.
\end{ex}

\begin{ex}
\label{ex:sat_compatible}
 Assume $p> h(G)$, where $h(G)$ denotes the Coxeter number of $G$.  The map ${\rm log}\colon \UU_G\to \NN_G$ from \cite[Thm.\ 3]{serre1} is a Springer map.  Let $H$ be a connected reductive subgroup of $G$.  We see that ${\rm log}$ is $H$-compatible if and only if $H$ is \emph{saturated} in the sense of \cite[Lecture 3]{serre1}.  For some properties of saturated subgroups, see \cite{serre1} and \cite{BBLMR}.
\end{ex}

\begin{ex}
\label{ex:no_good}
 Let $G= \SL_2\times \SL_2$.  For $q$ a positive power of $p$, let $H_q$ be $\SL_2$ diagonally embedded in $G$ with a $q$-Frobenius twist in one of the factors: say, the second factor.  Note that $\Lie(H_q)= \Lie(\SL_2)\oplus 0$, so $\Lie(H_q)$ contains no nilpotent elements that are distinguished in $\frakg$.  It follows from Remark~\ref{rem:distinguished_pres} that no Springer map for $G$ is $H_q$-compatible, so $H_q$ is not Springer-compatible.
 
 We can find a similar example for $G$ simple.  Let $G$ be a simple group of type $G_2$ and assume $p> 2$.  Define $H_q$ to be $\SL_2$ diagonally embedded in the $A_1\widetilde{A_1}$ regular subgroup of $G$ with a $q$-Frobenius twist in one of the factors, and let $1\neq u\in H_q$ be unipotent.  Then $u$ is a distinguished unipotent element of $G$ by \cite[Table 10, \S 4.1]{lawther2}, but $\Lie(H_q)$ contains no nilpotent elements that are distinguished in $\frakg$, so $H_q$ is not Springer-compatible.  We are grateful to Adam Thomas for this example.
\end{ex}

\begin{lem}
\label{lem:all_product_Springer}
 Let $G_1, G_2$ be connected reductive groups and let $\phi$ be a Springer map for $G_1\times G_2$.  Then $\phi$ is $G_1$-compatible and $G_2$-compatible.  Moreover, $\phi= \phi_1\times \phi_2$, where $\phi_i$ is the restriction of $\phi$ to $G_i$.
\end{lem}

\begin{proof}
 By Remark~\ref{rem:Springer_product}, we can reduce to the case when $G_1\times G_2$ is of separable type.  The $G_i$-compatibility of $\phi$ follows easily from the $(G_1\times G_2)$-equivariance.  Now fix regular $u_1\in \UU_{G_1}$ and $u_2\in \UU_{G_2}$, and set $X= (X_1,X_2)$, where $X_i= \phi_i(u_i)$ for $1\leq i\leq 2$.  Then $X_i$ is a regular element of $\Lie(G_i)$ for $1\leq i\leq 2$, $u=(u_1,u_2)$ is a regular element of $G$ and $X$ is a regular element of $\Lie(G)$.  Clearly $C_{G_i}(u_i)= C_{G_i}(X_i)$ for $1\leq i\leq 2$.
 
 Let $\phi_i'$ be the unique Springer isomorphism for $G_i$ such that $\phi'_i(u_i)= X_i$.  We have $(\phi_1'\times \phi_2')((u_1, u_2))= (\phi_1'(u_1), \phi_2'(u_2))= (X_1, X_2)= \phi((u_1, u_2))$, so $\phi= \phi_1'\times \phi_2'$.  Moreover, $(\phi_1(u_1), 0)= \phi((u_1, 0))= (\phi_1'\times \phi_2')((u_1, 0))= (\phi_1'(u_1), 0)$, so $\phi_1(u_1)= \phi_1'(u_1)$, so $\phi_1= \phi_1'$.  Likewise $\phi_2= \phi_2'$, and the result follows.
\end{proof}

\subsection{Cocharacters associated to nilpotent and unipotent elements}
\label{sub:cochars}
The Jacobson-Morozov Theorem allows one to associate an 
$\mathfrak{sl}(2)$-triple to any given non-zero element of $\NN_G$
in characteristic zero or large positive characteristic.
This is an indispensable tool in the 
Dynkin-Kostant classification of the
nilpotent orbits in characteristic zero as well as 
in the Bala-Carter classification of unipotent conjugacy classes of $G$
in large prime characteristic, see \cite[\S 5.9]{carter:book}.
In good characteristic there is a replacement for
$\mathfrak{sl}(2)$-triples, so-called \emph{associated 
	cocharacters}; see Definition \ref{def:assoc} below. These cocharacters are important tools in the classification theory of 
unipotent classes and nilpotent orbits of reductive algebraic groups
in good characteristic, see
for instance \cite[\S 5]{Jantzen} and \cite{premet}. We recall the relevant concept 
of cocharacters associated to a nilpotent element following 
\cite[\S 5.3]{Jantzen}.

\begin{defn}
\label{def:assoc}
	Let $X\in \NN_G$.  A cocharacter $\lambda \in Y(G)$ of $G$ is 
	\emph{associated} to $X$ (in $G$) provided
	$X \in \frakg(2,\lambda)$ and 
	there exists a Levi subgroup $L$ of $G$
	such that $X$ is distinguished nilpotent in $\Lie(L)$ and 
	$\lambda(\BBG_m) \leq \DD L$.  Following \cite[Def.~2.13]{RF}, we write
 \[
\Omega^a_G(X)  := \{\lambda \in Y(G)\mid 
\lambda  \text{ is associated to } X  \}
\]
for the set of cocharacters of $G$ associated to $X$.  Likewise, for $M$ a connected reductive subgroup of $G$ such that $X\in \Lie(M)$, we write  $\Omega^a_M(X)$ for the set of cocharacters of $M$ that are associated to $X$.
This notation stems from the fact that associated cocharacters are destabilising cocharacters of $G$ for $X$ in the sense of Kempf-Rousseau theory, see \cite{premet} and \cite{mcninch2}.
	
Let $u\in \UU_G$.  A cocharacter $\lambda \in Y(G)$ of $G$ is 
	\emph{associated} to $u$ (in $G$) provided it is associated to $\phi(u)$, 
	where $\phi : \UU_G \to \NN_G$ is a fixed Springer map as in \S \ref{sec:springer}; see \cite[\S~3]{mcninchsommers}. We write
 \[
\Omega^a_{G,\phi}(u) := \{\lambda \in Y(G)\mid 
\lambda  \text{ is associated to } u  \}
\]
for the set of cocharacters of $G$ associated to $u$.  Likewise, for $M$ a connected reductive subgroup of $G$ containing $u$ and $\phi'$ a Springer map for $M$, we write $\Omega^a_{M,\phi'}(u)$ for the set of cocharacters of $M$ that are associated to $u$.  If $\phi$ is understood then we sometimes write $\Omega^a_G(u)$ instead of $\Omega^a_{G,\phi}(u)$.
%If necessary we write $\Omega^a_{G,\phi}(u)$ instead of $\Omega^a_G(u)$, to indicate the Springer map. 
\end{defn} 

\begin{rem}
\label{rem:cochar1}
Let $u \in \UU_G$,  
$\lambda \in \Omega^a_G(u)$, and $g \in C_G(u)$. 
Then $g \cdot \lambda$ is also associated to $u$, see \cite[\S 5.3]{Jantzen}. 
Proposition \ref{prop:cochar1}(ii) gives a converse to this property. 
\end{rem}

\begin{rem}
\label{rem:assoc_reduction}
 Let $G_1,\ldots, G_r$ be connected reductive groups and set $G= G_1\times\cdots\times G_r$.  Let $u_i\in \UU_{G_i}$ for each $1\leq i\leq r$ and let $L$ be a Levi subgroup of $G$.  Then $L= L_1\times\cdots\times L_r$ for some Levi subgroups $L_i$ of $G_i$.  Set $u= (u_1,\ldots, u_r)\in L$.  It is clear that $u$ is distinguished in $L$ if and only if $u_i$ is distinguished in $L_i$ for each $i$.  Likewise, if $X= (X_1,\ldots, X_r)\in \NN_L$ then $X$ is distinguished in $\Lie(L)$ if and only if $X_i$ is distinguished in $\Lie(L_i)$ for each $i$.
 
 Fix a Springer map for $G$.   Let $\lambda\in Y(G)$.  We can write $\lambda= \lambda_1\times\cdots \times \lambda_r$ for some $\lambda_i\in Y(G_i)$.  It follows from the previous paragraph that $\lambda$ is associated to $X$ in $\Lie(G)$ if and only if $\lambda_i$ is associated to $X_i$ in $\Lie(G_i)$ for each $i$, \cite[\S 5.6]{Jantzen}.  We deduce the analogous statement for $u$ from Remark~\ref{rem:Springer_product}: if $\phi= \phi_1\times\cdots\times \phi_r$ is a Springer map for $G$ then $\lambda$ is associated to $u$ in $G$ if and only if $\lambda_i$ is associated to $u_i$ in $G_i$ for each $i$.
 
 Let $\psi\colon \widetilde{G}\to G$ be an epimorphism of connected reductive groups such that $\ker(d\psi)$ is central in $\Lie(\widetilde{G})$.  Let $\widetilde{u}\in \UU_{\widetilde{G}}$, let $\widetilde{X}\in \NN_{\widetilde{G}}$, let $\widetilde{L}$ be a Levi subgroup of $\widetilde{G}$ and let $\widetilde{\lambda}\in Y(\widetilde{G})$.  Set $u= \psi(\widetilde{u})$, $X= d\psi(\widetilde{X})$, $L= \psi(\widetilde{L})$ and $\lambda= \psi\circ \widetilde{\lambda}$.  Let $\widetilde{\phi}$ be a Springer map for $\widetilde{G}$ and let $\phi$ be the corresponding Springer map for $G$ as described in Remark~\ref{rem:springer_homeom}.  Using \cite[\S 4.3]{Jantzen} and Remark~\ref{rem:Springer_product} we get analogues of the above statements: $u$ is distinguished in $L$ if and only if $\widetilde{u}$ is distinguished in $\widetilde{L}$, $X$ is distinguished in $\Lie(L)$ if and only if $\widetilde{X}$ is distinguished in $\Lie(\widetilde{L})$ and $\lambda$ is associated to $X$ (resp., to $u$) if and only if $\widetilde{\lambda}$ is associated to $\widetilde{X}$ (resp., to $\widetilde{u}$).
% [Analogous results for distinguished parabolics as well?]
\end{rem}

\begin{rem}
\label{rem:assoc_indpnce}
 The notion of an associated cocharacter for an element $u\in \UU_G$ depends on the choice of the Springer map for $G$: see \cite[Rem.\ 23]{mcninch2}.  We do, however, have the following.  Let $\phi_1$ and $\phi_2$ be Springer maps for $G$.  Let $1\neq u_1\in \UU_G$ and let $\lambda\in \Omega^a_{G, \phi_1}(u_1)$.  Then $\lambda\in \Omega^a_{G, \phi_2}(u_2)$, where $u_2= \phi_2^{-1}(\phi_1(u))$.  Note that $u_2$ is conjugate to $u_1$ by Remark~\ref{rem:conj_class_map}.
\end{rem}

We require some basic facts about cocharacters associated to unipotent elements. 
The following results are \cite[Lem.\ 5.3; Prop.~5.9]{Jantzen}  for nilpotent elements (see also \cite[Thm.\ 2.3, Prop.\ 2.5]{premet}); the versions for unipotent elements follow immediately.

\begin{prop}
	\label{prop:cochar1} 
	Suppose $p$ is good for $G$.  Let $1 \ne u \in \UU_G$.	
	\begin{itemize} 	
			\item[(i)] 
		$\Omega^a_G(u) \ne \varnothing$, i.e.,
		cocharacters of $G$ associated to $u$ exist.
		\item[(ii)] $C_G(u)^{\circ}$ acts transitively on $\Omega^a_G(u)$.
		\item[(iii)] 
		Let $\lambda \in \Omega^a_G(u)$ and let $P_\lambda$ be the parabolic subgroup of $G$ defined by $\lambda$ as in \S \ref{sub:parabolics}. Then $P_\lambda$ depends only on $u$ and not on the choice of $\lambda$. 	
		\item[(iv)] 
			Let $\lambda \in \Omega^a_G(u)$ and let $P(u) := P_\lambda$ be as in (iii). Then $C_G(u) \subseteq P(u)$. 
		\end{itemize} 
\end{prop}

If $u$ is distinguished in $G$, then the parabolic subgroup $P(u)$ of $G$ from Proposition \ref{prop:cochar1}(iii) is a distinguished parabolic subgroup of $G$ and $u$ belongs to the Richardson orbit of $P(u)$ on its unipotent radical, see Theorem \ref{thm:bala-carter}(i); see also \cite[Prop.~22]{mcninch2}.

\begin{rem}
	\label{rem:cocharinA1}
	Let $p > 0$ %is good for $G$ 
	 and suppose $1 \ne u \in \UU_G^{(1)}$ is
	 contained in a subgroup $A$ of $G$ of type $A_1$.   Such a subgroup $A$ always exists when $p$ is good, and when $p$ is bad there is essentially only one exception, due to Testerman \cite{testerman} and  Proud-Saxl-Testerman \cite{PST} --- see Theorems \ref{thm:A1pgood} and \ref{thm:A1pbad}.
	Then, since $p$ is good for $A$, by Proposition \ref{prop:cochar1}(i) there exists a cocharacter 
	$\lambda \in \Omega^a_A(u)$. Note that
	$\lambda(\BBG_m)$ is a maximal torus in $A$. 
\end{rem}

It follows from the work of Pommerening 
\cite{pommereningI}, \cite{pommereningII} that the description of the unipotent classes in characteristic $0$ is identical to the one for $G$ when $p$ is good for $G$. In both instances these are described by so-called  \emph{weighted Dynkin diagrams}. 
As a result, a cocharacter associated to a unipotent element in good characteristic acts with the same 
weights on the Lie algebra of $G$ as its counterpart does in characteristic $0$.
This fact is used in the proof of the following result by Lawther \cite[Thm.~1]{lawther}; see also the proof of \cite[Prop.~4.2]{seitz} and \cite[Rem.\ 31]{mcninch2}.  The result is stated in {\em loc.\ cit.}\ for $G$ simple, but the extension to arbitrary connected reductive $G$ is immediate, using arguments like those in Remark~\ref{rem:assoc_reduction}; note that if $\psi\colon \widetilde{G}\to G$ is an epimorphism of connected reductive groups such that $\ker(d\psi)$ is central in $\Lie(G)$ then $d\psi$ gives an isomorphism from $\Lie(\widetilde{U})$ onto $\Lie(\psi(\widetilde{U}))$, where $\widetilde{U}$ is any maximal unipotent subgroup of $\widetilde{G}$, so the weights of $\widetilde{\lambda} \in Y(\widetilde{G})$ on $\Lie(\widetilde{G})$ are the same as the weights of $\psi\circ\widetilde{\lambda}$ on $\Lie(G)$.

\begin{lem}
	\label{lem:lawther}
	Let $u\in \UU_G$. Suppose $p$ is good for $G$.  Let $\lambda\in \Omega^a_G(u)$. Denote by $\omega_G$ the highest weight of $\lambda(\BBG_m)$ on $\frakg$. Then $u$ has order $p$ if and only if $\omega_G\leq 2p-2$.
\end{lem}

The concept of associated cocharacters is not only a
convenient replacement for 
$\mathfrak{sl}(2)$-triples 
from the 
Jacobson-Morozov theory, it is a very powerful tool in the 
classification theory of unipotent conjugacy classes and nilpotent orbits. 
Specifically, 
in \cite{premet}
Premet showcases a conceptual and uniform proof of Pommerening's extension of the Bala-Carter Theorem 
\ref{thm:bala-carter}
to good characteristic. 
His proof uses the fact that associated cocharacters are \emph{optimal} in the geometric invariant theory sense of Kempf-Rousseau-Hesselink.

\subsection{Cocharacters associated to distinguished elements}
\label{sub:cochars_dist}

The linchpin of our 
proofs of Theorems \ref{thm:dist-orderp} and \ref{thm:dist-orderp-finite} is the following collection of facts.

\begin{lem}[{\cite[Lem.~3.1]{RF}}] 
	\label{lem:RF_Lie}
	Suppose $p$ is good for $G$. 
	Let $M$ be a connected reductive subgroup of $G$. Let $X\in \Lie(M)$ be a distinguished nilpotent element of $\frakg$. 
	Then $\Omega^a_M(X) = \Omega^a_G(X) \cap Y(M)$.
\end{lem}

The assertion of the lemma fails in general if $X$ is not distinguished in $\frakg$, even when $p$ is good for both $M$ and $G$: e.g., see \cite[Rem.~5.12]{Jantzen}. However, we do have the following result for all nilpotent elements in good characteristic.

\begin{lem}[{\cite[Cor.~3.22]{RF}}] 
	\label{lem:RFlevis_Lie}
	Suppose $p$ is good for $G$. 
	Let $L\subset G$ be a Levi subgroup of $G$. Let $X\in \NN_L$. 
	Then $\Omega^a_L(X) = \Omega^a_G(X) \cap Y(L)$.
\end{lem}

We need group-theoretic analogues of Lemmas~\ref{lem:RF_Lie} and \ref{lem:RFlevis_Lie}.  For the former we need an extra Springer compatibility assumption, otherwise the result can fail (see Remark~\ref{rem:daggerless}).

\begin{lem}
	\label{lem:RF}
	Suppose $p$ is good for $G$. 
	Let $M$ be a connected reductive subgroup of $G$.  Suppose $M$ is Springer-compatible and let $\phi$ be an $M$-compatible Springer map.  Let $u\in M$ be a distinguished unipotent element of $G$. 
	Then $\Omega^a_{M, \phi_M}(u) = \Omega^a_{G, \phi}(u) \cap Y(M)$.
\end{lem}

\begin{proof}
 Let $X= \phi(u)= \phi_M(u)$.  Then
 $$ \Omega^a_{M, \phi_M}(u)= \Omega^a_M(X)= \Omega^a_G(X)\cap Y(M)= \Omega^a_{G, \phi}(u) \cap Y(M), $$
 where the middle equality is from Lemma~\ref{lem:RF_Lie}.
\end{proof}

\begin{lem} 
	\label{lem:RFlevis}
	Suppose $p$ is good for $G$.   
	Let $L\subset G$ be a Levi subgroup of $G$ and let $\phi$  be a Springer map for $G$.  Let $u\in \UU_L$. 
	Then $\Omega^a_{L,\phi_L}(u) = \Omega^a_{G,\phi}(u) \cap Y(L)$.
\end{lem}

\begin{proof}
 Since $L= C_G(S)$ for some torus $S$, $\phi$ is $L$-compatible by Example~\ref{ex:Springer_cent}.  The result now follows by the same argument as in Lemma~\ref{lem:RF}.
\end{proof}

\section{Good $A_1$ subgroups}
\label{sec:A1s}

\subsection{Good $A_1$ overgroups}
\label{sec:goodA1overgroups}

In his seminal work \cite{seitz}, Seitz defines an important class of $A_1$ overgroups of an element $1 \ne u\in \UU_G^{(1)}$ for $G$ simple (see \cite[Sec.\ 1]{seitz}).  He 
establishes the existence and fundamental properties of these overgroups provided $p$ is good for $G$.  We recall some of these results and generalise them to arbitrary connected reductive $G$.

\begin{defn}
	\label{def:goodA1}
	Following \cite[\S  1]{mcninch1}, we say that a homomorphism $\beta\colon \SL_2\to G$ is \emph{good} if each weight of the corresponding representation of $\SL_2$ on $\frakg$ is at most $2p-2$.  We say that a subgroup $A$ of $G$ of type $A_1$ is a \emph{good $A_1$} subgroup of $G$, or is \emph{good for $G$}, if it is the image of a good homomorphism.   
	Else we call $A$ a \emph{bad $A_1$} subgroup of $G$.
	This is of course independent of the choice of a maximal torus of $A$.
For $1 \ne u\in \UU_G^{(1)}$, we define 
\[
\AA(u) := \AA_G(u) := \{A \subseteq G \mid \text{$A$ is a good $A_1$ subgroup of $G$ containing $u$}\}
\]
and analogously, for a connected reductive subgroup  $M$ of $G$ we write $\AA_M(u)$ for the set of all good $A_1$ subgroups of $M$ containing $u$.

Clearly any conjugate of a good $A_1$ homomorphism (resp., subgroup) is good.  If $A \subseteq H \subseteq G$ are connected reductive groups such that $A$ is a good $A_1$ subgroup of $G$, then $A$ is obviously also a good $A_1$ subgroup of $H$.  We see in Lemma~\ref{lem:good_ascent} that the converse holds under some extra hypotheses.  The converse is false in general, however: e.g., just take $A= H$ to be a bad $A_1$ subgroup of $G$.
\end{defn}

\begin{ex}
    \label{ex:goodA1inGLV}
    Let $V$ be an $\SL_2$-module such that weights of a maximal torus $T$ of $\SL_2$ on $V$ are less than $p$.  Then the weights of $T$ in the induced action on $\Lie(\GL(V))  \cong V \otimes V^*$ are at most $2p-2$. Thus
the induced subgroup $A$ in $\GL(V)$ is a good $A_1$. 
In this situation the highest weights of $T$ on each composition factor of $V$ are restricted, so $V$ is a semisimple $\SL_2$-module; see \cite[Cor.~3.9]{AJL}.  Hence $A$ is $\GL(V)$-cr; this is a special case of Theorem~\ref{thm:seitz1}(iii) below.
\end{ex}

We record parts of the main theorems from \cite{seitz} for our purposes, using the notation above. These were formulated and proved in \emph{loc.~cit.}\
for simple $G$, but we need extensions to arbitrary connected reductive $G$. To obtain this, we need the following lemma.

\begin{lem}\label{lem:goodhoms}
	Let $G$ be a connected reductive group. Let $\beta_1,\beta_2:\SL_2\to G$ be good homomorphisms with the same image $A$. Then $\beta_1$ and $\beta_2$ are conjugate by an element of $A$.
\end{lem}

\begin{proof}
	Assume first that $A\cong\SL_2$.  Let $1\leq i\leq 2$.  Then we can regard $\beta_i$ as an element of $\End(\SL_2)$, so it is an inner endomorphism followed by a Frobenius $q$th power map for $q=p^r$ for some $r\geq0$.  Let $T$ be a maximal torus of $\SL_2$.  If $r\geq 1$ then the highest weight of $T$ is at least $2q$, since $\SL_2$ acts on $\Lie(A)$ with highest weight $2$, which contradicts the goodness assumption. Therefore, $\beta_i\in \Aut(\SL_2)$.  But all automorphisms of $\SL_2$ are inner.  The result follows.
	
 For $A\cong \PGL_2$, we can factor $\beta_i$ as $\SL_2\to \PGL_2\stackrel{\beta_i'}{\to}\PGL_2$, where the first map is the canonical projection. One can now apply an argument like the above one to the maps $\beta_i'\colon \PGL_2\to A$.
\end{proof}
	
\begin{thm}
	\label{thm:seitz1}
	Let $G$ be connected reductive.
	Suppose $p$ is good for $G$ and let $1\neq u\in \UU_G^{(1)}$.  Then the following hold: 
	\begin{itemize}
		\item [(i)] $\AA(u) \ne \varnothing$.
		\item [(ii)] $R_u(C_G(u))$ acts transitively on $\AA(u)$. 
		\item [(iii)]
  Let $A \in \AA(u)$. Then $A$ is $G$-completely reducible.
                 \item [(iv)] There is a unique 1-dimensional unipotent subgroup $U$ of $G$ such that $u\in U$ and $U$ is contained in a good $A_1$ subgroup of $G$.
	\end{itemize}
\end{thm}

\begin{notn}
 We denote the subgroup $U$ from Theorem~\ref{thm:seitz1}(iv) by $\mathcal{U}(u)$.
\end{notn}
 
\begin{proof}[Proof of Theorem~\ref{thm:seitz1}]
	For $G$ simple see \cite[Thms.~1.1--1.3]{seitz}. Now let $G$ be connected reductive.  Since $\SL_2$ and $\PGL_2$ are perfect, any $A_1$ subgroup of $G$ is contained in $\DD G$.  Hence without loss we can assume that $G$ is semisimple; note for (iii) that a subgroup of $\DD G$ is $\DD G$-cr if and only if it is $G$-cr \cite[Prop.\ 2.8]{BMR:commuting}.  Moreover, let $\psi\colon \widetilde{G}\to G$ be a central isogeny of connected reductive groups.  If $\widetilde{A}$ is an $A_1$ subgroup of $\widetilde{G}$ then $\widetilde{A}$ is good for $\widetilde{G}$ if and only if $\psi(\widetilde{A})$ is good for $G$: cf.\ the argument of the paragraph preceding Lemma~\ref{lem:lawther}.  Note also for (iii) that if $\widetilde{H}\subseteq \widetilde{G}$ then $\widetilde{H}$ is $\widetilde{G}$-cr if and only if $\psi(\widetilde{H})$ is $G$-cr \cite[Lem.\ 2.12]{BMR}.  Hence we can assume without loss that $G=G_1\times \cdots \times G_r$, where each $G_i$ is simple.

 We need a description of good $A_1$ subgroups of $G$ in terms of good $A_1$ subgroups of the $G_i$.  Let $T$ be a maximal torus of $\SL_2$.  Denote by $\pi_i$ the projection from $G$ to $G_i$. Let $\beta:\SL_2\to G$ be a homomorphism and define $\beta_i:=\pi_i\circ\beta$. For notational convenience, we assume that each $\beta_i$ is non-trivial.  The weights of $T$ on $\Lie(G_i)$ form a subset of the set of weights of $T$ on $\Lie(G)$, since $\Lie(G)=\oplus \Lie(G_i)$. Therefore, if $\beta$ is a good homomorphism for $G$, then $\beta_i$ is a good homomorphism for $G_i$ or trivial.
 Conversely, if $\beta_i:\SL_2\to G_i$ is a non-trivial homomorphism for each $i$, define $\beta:=\beta_1\times\cdots\times \beta_r$ to be the diagonal embedding into $G$. Then the maximal weight $\omega_G$ of $T$ on $\Lie(G)$ is given by $\max\{\omega_{G_i}\}$, where $\omega_{G_i}$ is the maximal weight of $T$ on $\Lie(G_i)$. Thus, $\beta$ is good if and only if the $\beta_i$ are good.  Now (i) and (iii) are immediate from the above observations, \cite[Lem.~2.12]{BMR} and the results for $G$ simple.
	
 For (ii), let $A^1$ and $A^2$ be good $A_1$ subgroups of $G=G_1\times\cdots\times G_r$ containing $u=(u_1,\ldots,u_r)$  with $u_i \ne 1$ for each $i$. Choose two good homomorphisms $\beta^1,\beta^2:\SL_2\to G$ such that $\Im(\beta^i)=A^i$.
% We can assume without loss that there exists $h\in \SL_2$ such that $\beta^1(h)= \beta^2(h)= u$.
 By the observations above, there are good homomorphisms $\beta_i^1,\beta_i^2:\SL_2\to G_i$ with images $A^1_i, A^2_i$ containing $u_i$. Now \cite[Thm.~1.1(ii)]{seitz} implies that $A_i^2= g_iA_i^1g_i^{-1}$ for some $g_i\in R_u(C_{G_i}(u_i))$.  Lemma \ref{lem:goodhoms} (applied to $G_i$) implies that $h_ig_i\cdot \beta_i^1= \beta_i^2$ for some $h_i\in A_i^2$.  Hence $\beta^2= hg\cdot \beta^1$, where $g= (g_1,\ldots, g_r)\in R_u(C_G(u))$ and $h= (h_1,\ldots, h_r)\in A^2$.  It follows that $A^2= g\cdot A^1$.
	
 For (iv), let $G=G_1\times\cdots\times G_r$ and let $u=(u_1,\ldots,u_r)\in \UU_G^{(1)}$ with $u_i \ne 1$ for each $i$. Choose an $A\in\AA_G(u)$ which is the image of the good homomorphism $\beta$. As before, we get good homomorphisms $\beta_i$ with images $A_i\in\AA_{G_i}(u_i)$, and $\beta= \beta_1\times\cdots\times \beta_r$.  Without loss we can assume the $\beta_i$ are non-trivial.  Fix a 1-dimensional unipotent subgroup $V$ of $\SL_2$.  After conjugating $\beta$ by an element of $A$, we can assume that $\mathcal{U}(u_i)=\beta_i(V)$ for each $i$. Define $\mathcal{U}(u)=(\beta_1\times\cdots\times\beta_r)(V)$. This is a $1$-dimensional unipotent subgroup of $G$ containing $u$ and is contained in the good $A_1$ subgroup $A$. This proves the existence. For the uniqueness, let $U'$ be another $1$-dimensional unipotent subgroup of $G$ such that $u\in U'\subseteq A'$ for some $A'\in\AA_G(u)$. By (ii), $A=gA'g\inverse$ for some $g\in C_G(u)$, and so $gU'g\inverse=\mathcal{U}(u)$. Write $g=(g_1,\ldots,g_r)$ with $g_i\in C_{G_i}(u_i)$. By \cite[Thm.\ 1.2(i)]{seitz} $g_i$ centralises $\mathcal{U}(u_i)$, hence $g$ centralises $\mathcal{U}(u)$. Thus, $U'=\mathcal{U}(u)$.
\end{proof}

\begin{ex}
 Let $H_q$ be the bad $A_1$ subgroup of $G= \SL_2\times \SL_2$ from Example~\ref{ex:no_good}.  Here $\beta_1=\id_{\SL_2}$, while $\beta_2:\SL_2\to \SL_2$ is the $q$th power map, which is not a good homomorphism.  On the other hand, the projection of $H_q$ onto each factor is just $\SL_2$, which is a good $A_1$ subgroup of $\SL_2$, so we cannot detect the badness of $H_q$ just by looking at its images in the simple factors of $G$.
\end{ex}

\begin{rem}
\label{rem:U_u}
 Let $1\neq u\in \UU_G^{(1)}$.  We claim that
 \begin{equation}
 \label{eqn:conn_red_case}
  C_G(\mathcal{U}(u))= C_G(u)= C_G(\Lie(\mathcal{U}(u))).
 \end{equation}
 To see this, suppose first that $\widetilde{G}$ is of the form $G_1\times\cdots\times G_r$, where each $G_i$ is simple, and let $\pi_i\colon \widetilde{G}\to G_i$ be the canonical projection.  Let $\widetilde{u}= (u_1,\ldots, u_r)\in \UU_{\widetilde{G}}^{(1)}$  with $u_i \ne 1$ for each $i$.  Choose a good homomorphism $\widetilde{\beta}\colon \SL_2\to \widetilde{G}$ such that $\mathcal{U}(\widetilde{u})\subseteq \widetilde{A}:= \Im(\SL_2)$, and set $\beta_i= \pi_i\circ \widetilde{\beta}$ and $A_i= \beta_i(\SL_2)$.  
 It follows from \cite[Thm.\ 1.2(i)]{seitz} that $C_{G_i}(\mathcal{U}(u_i))= C_{G_i}(u_i)= C_{G_i}(\Lie(\mathcal{U}(u_i)))$ for each $i$.  We deduce from the arguments in the proof of Theorem~\ref{thm:seitz1} that $C_{\widetilde{G}}(\mathcal{U}(\widetilde{u}))= C_{\widetilde{G}}(\widetilde{u})= C_{\widetilde{G}}(\Lie(\mathcal{U}(\widetilde{u})))$; note that $d\beta_i\colon \Lie(\SL_2)\to \Lie(A_i)$ is surjective for each $i$ because $\beta_i$ does not involve a Frobenius twist.
 
 If $\psi\colon \widetilde{G}\to G$ is a central isogeny and $1 \ne \widetilde{u}\in \UU_{\widetilde{G}}^{(1)}$, then it is clear that $\mathcal{U}(u)= \psi(\mathcal{U}(\widetilde{u}))$, where $u= \psi(\widetilde{u})$, and we deduce that
 \begin{equation}
 \label{eqn:ss_case}
  C_G(\mathcal{U}(u))= C_G(u)= C_G(\Lie(\mathcal{U}(u))).
 \end{equation}
 
 Now let $G$ be an arbitrary connected reductive group and let $1 \ne u\in \UU_G^{(1)}$.  Then $\mathcal{U}(u)\subseteq \DD G$.  Now \eqref{eqn:conn_red_case} follows easily from \eqref{eqn:ss_case} applied to the semisimple group $\DD G$.  

 We deduce from \eqref{eqn:conn_red_case} and Theorem~\ref{thm:seitz1}(ii) that $\mathcal{U}(u)$ is contained in {\bf every} good $A_1$ overgroup of $u$.
\end{rem}

\begin{lem}
\label{lem:U_u_overgroup}
 Suppose $p$ is good for $G$.  Let $1\neq u\in \UU_G^{(1)}$ and let $A$ be an $A_1$ subgroup of $G$ containing $\mathcal{U}(u)$.  Then $A$ is good in $G$.
\end{lem}

\begin{proof}
 Let $A'$ be a good $A_1$ subgroup containing $\mathcal{U}(u)$.  Then $A$ and $A'$ have a common maximal unipotent subgroup $\mathcal{U}(u)$.  By \cite[Thm.\ 1.1]{londmartin}, $A$ and $A'$ are $G$-conjugate.  Hence $A$ is good, because $A'$ is.
\end{proof}

\begin{lem}
\label{lem:good_crit}
 Suppose $p$ is good for $G$.  Let $A$ be an $A_1$ subgroup of $G$ and let $\lambda\in Y(A)$.  Suppose that
 \begin{itemize}
  \item[(i)] $\lambda\in \Omega^a_G(X)$ for some $0\neq X\in \NN_G^{(1)}$, or
  \item[(ii)] $\lambda\in \Omega^a_{G,\phi}(u)$ for some $1\neq u\in \UU_G^{(1)}$ and some Springer map $\phi$ for $G$.
  \end{itemize}
  Then $A$ is a good $A_1$ subgroup of $G$.
\end{lem}

\begin{proof}
 Let $\phi$ be a Springer map for $G$, and suppose $\lambda\in \Omega^a_{G,\phi}(u)$ for some $1\neq u\in \UU_G^{(1)}$.  It follows from Lemma~\ref{lem:lawther} that the weights of $\lambda$ on $\frakg$ are at most $2p- 2$.  Define $\beta\colon \SL_2\to A$ to be an isomorphism if $A\iso \SL_2$, and the usual central isogeny $\SL_2\to \PGL_2$ followed by an isomorphism from $\PGL_2$ onto $A$ if $A\iso \PGL_2$.  Then there exists $\mu\colon \Gm\to \SL_2$ such that $\mu$ is an isomorphism onto a maximal torus of $\SL_2$ and $\lambda= \beta\circ \mu$.  The weights of $\mu$ on $\frakg$ are at most $2p- 2$ by construction, so $A$ is good.  Hence $A$ is good if (ii) holds.
 
 If (i) holds then $\lambda\in \Omega^a_{G,\phi}(u)$, where $u:= \phi^{-1}(X)$.  But $u\in \UU_G^{(1)}$, by Lemma~\ref{lem:p_preserve}, so (ii) holds, so $A$ is good by the argument above.
\end{proof}

In the next theorem we recall parts of the analogue of Theorem \ref{thm:seitz1} for finite overgroups of type $A_1$.

\begin{thm}
	\label{thm:seitz2}
	Let $G$ be connected reductive.
	Suppose $p$ is good for $G$. Let $\sigma : G \to G$ be a Steinberg endomorphism of $G$. Suppose $u \in G_\sigma$  is unipotent of order $p$.
	\begin{itemize}
		\item [(i)] $\AA(u)_\sigma \ne \varnothing$.
		\item [(ii)] 	$R_u(C_G(u))_\sigma$ acts transitively on $\AA(u)_\sigma$.
  \item [(iii)] 
		Let $A \in \AA(u)_\sigma$. Suppose that $q> 7$ if $G$ is of exceptional type.  Then $A_\sigma$ is $\sigma$-completely reducible.
                 \item [(iv)] There is a unique $\sigma$-stable 1-dimensional unipotent subgroup $U$ of $G$ such that $u\in U$ and $U$ is contained in a good $A_1$ subgroup of $G$.
	\end{itemize}
\end{thm}

\begin{proof}
 (i)--(iii): The simple case is proved by Seitz in \cite[Thm.~1.4]{seitz}. For connected reductive groups we use an argument similar to the one in the proof of Theorem \ref{thm:seitz1}.
 
 (iv) By (i) we can choose some $A\in \AA(u)_\sigma$.  Now $\mathcal{U}(u)\subseteq A$ by Remark~\ref{rem:U_u}.  Clearly $\mathcal{U}(u)$ is the unique 1-dimensional unipotent subgroup of $A$ that contains $u$, so $\mathcal{U}(u)$ must be $\sigma$-stable.  Hence $U:= \mathcal{U}(u)$ has the desired properties.
\end{proof}

\begin{rem}
	\label{rem:exsigmaA1}
	Parts (i) and (ii) of Theorem \ref{thm:seitz2} follow from parts (i) and (ii) of Theorem \ref{thm:seitz1} and the Lang-Steinberg Theorem,
	see \cite[Prop.~9.1]{seitz}.
\end{rem}

\begin{rem}
    \label{rem:seitz}
(i).   Concerning the terminology in Theorem \ref{thm:seitz2}(iii), following \cite{HRG}, a subgroup $H$ of $G$ is said to be  
\emph{$\sigma$-completely reducible}, provided that whenever 
$H$ lies in a $\sigma$-stable parabolic subgroup 
$P$ of $G$, it lies in a $\sigma$-stable Levi subgroup of $P$.
This notion is motivated by certain rationality questions concerning $G$-complete reducibility; see \cite{HRG} for details. For a $\sigma$-stable subgroup 
$H$ of $G$, this property is equivalent to $H$ being $G$-cr, 
thanks to \cite[Thm.~1.4]{HRG}.

(ii).
Apart from the special conjugacy class of good $A_1$ subgroups in $G$
asserted in Theorem \ref{thm:seitz1}, 
there might be a plethora of  conjugacy classes of bad $A_1$ subgroups in $G$ even when $p$ is good for $G$. 
Just take a non-semisimple representation $\beta : \SL_2 \to \SL(V) = G$ in characteristic $p > 0$. Then the $A_1$ subgroup 
$\beta(\SL_2)$ is bad in $G$, while $p$ is good for $G$. For a concrete example, see \cite[Rem.~5.12]{Jantzen}. This can only happen if $p$ is sufficiently small compared to the rank of $G$, thanks to Theorem \ref{thm:ag}.

The subgroups $H_q$ of $\SL_2\times \SL_2$ in Example~\ref{ex:no_good} are also bad $A_1$ subgroups: see Remark~\ref{rem:daggerless}.

(iii).
The proofs of Theorems \ref{thm:seitz1} and \ref{thm:seitz2} for $G$ simple
by Seitz in \cite{seitz} depend on separate considerations for each Dynkin type and involve
in part intricate arguments for the component groups of centralizers of unipotent elements. 
In  \cite{mcninch2}, McNinch presents 
uniform proofs of Seitz's theorems for $G$ strongly standard reductive, which are almost entirely free of any case-by-case checks, utilizing methods from geometric invariant theory. However, McNinch's argument (see \cite[Thm.~44]{mcninch2}) of the
conjugacy result in Theorem \ref{thm:seitz1}(ii) 
depends on the fact that for a good $A_1$ subgroup $A$ of $G$, the $A$-module $\frakg$ is a tilting module. The latter is established by Seitz in \cite[Thm.~1.1]{seitz}.
\end{rem}

In \cite[\S 9]{seitz}, Seitz exhibits instances when there is no good $A_1$ overgroup of an element of order $p$ when $p$ is bad for $G$. 
As we explain next, Example \ref{ex:c2} gives a counterexample to Theorem \ref{thm:seitz1}(iii) in case $p$ is bad for $G$: that is, it gives a good $A_1$ subgroup $A$ such that $A$ is not $G$-cr.  Specifically, we show that some of the $A_1$ subgroups in that example are good $A_1$ subgroups of $G$, but thanks to Example \ref{ex:c2}, they are not $G$-cr. 

\begin{ex}[Example \ref{ex:c2} continued]
	\label{ex:goodA1}
Let $G$ be simple of type $C_2$ and $p = 2$.
Let $\sigma$ be $\id_G$ or a $q$-Frobenius endomorphism of $G$. 
Let $\mathscr C$ denote the subregular unipotent class of $G$.
Suppose $u \in \mathscr C \cap G_\sigma$.
Then by Example \ref{ex:c2} there are 
$\sigma$-stable subgroups $A$ of $G$ of type $A_1$ containing $u$ that are not $G$-cr.
Specifically, let $E$ be the natural module for $\SL_2$. Consider the two conjugacy classes of embeddings of $\SL_2$ into $G = \Sp(V)$, where we take either $V \cong E \perp E$ or $V \cong E \otimes E$, as an $\SL_2$-module.
The images of both embeddings meet the class $\mathscr C$ non-trivially.
One checks that the highest weight
of a maximal torus of $\SL_2$ on $\frakg$ is $4$ in the second instance. So in this case the image of $\SL_2$ in $G$ is not a good $A_1$.
In contrast, in the first instance 
the highest weight of a maximal torus of $\SL_2$ on $\frakg$ is $2 = 2 p -2$, by Example \ref{ex:goodA1inGLV}. So the image  of $\SL_2$  in $G$ is a good $A_1$ in $\SL(V)$, and so it is a good $A_1$ in $G$ as well. 
\end{ex}

\subsection{Characterisations of good $A_1$ subgroups}
\label{sec:good_characterisations}
In this section we investigate some other types of $A_1$ subgroup which were introduced by McNinch.  We prove that these other notions are all equivalent to goodness (Theorem~\ref{thm:good_equivalences}).  The key ingredient we need is work of Sobaje, who proved the existence of a Springer map for $G$ with especially nice properties.  We assume throughout the section that $p$ is good for $G$.

 We recall a construction from \cite[Prop.\ 5.2]{seitz} (see also \cite{sobaje}).  Let $P$ be a parabolic subgroup of $G$, and set $U= R_u(P)$.  It can be shown that any Springer map for $G$ maps $U$ to $\Lie(U)$.  Suppose $U$ has nilpotency class less than $p$; in this case we say that $P$ is \emph{restricted}.  In particular, any distinguished parabolic subgroup of $G$ corresponding to a distinguished unipotent element of order $p$ is restricted \cite[Prop.\ 24]{mcninch2}.  We endow $\Lie(U)$ with the structure of an algebraic group using the Baker-Campbell-Hausdorff formula.  There is a unique $P$-equivariant isomorphism of algebraic groups ${\rm exp}_P\colon \Lie(U)\to U$ such that the derivative of ${\rm exp}_P$ is the identity on $\Lie(U)$ (this is established in \cite[Prop.\ 5.2]{seitz} for semisimple $G$, but the extension to connected reductive $G$ is immediate).  We denote the inverse of ${\rm exp}_P$ by ${\rm log}_P\colon \Lie(U)\to U$.

\begin{defn}
\label{defn:logarithmic}
 We say that a Springer map $\phi$ for $G$ is \emph{logarithmic} if the following holds: for any $1 \ne u\in \UU_G^{(1)}$, the restriction of $\phi$ gives an isomorphism $\phi_u$ of algebraic groups from $\mathcal{U}(u)$ to $\Lie(\mathcal{U}(u))$, and $d\phi_u$ is the identity on $\Lie(\mathcal{U}(u))$.
\end{defn}

\begin{prop}
\label{prop:logarithmic_springer}
 \begin{itemize}
  \item[(i)] There exists a logarithmic Springer map for $G$.
  \item[(ii)] Let $\phi$ be a logarithmic Springer map for $G$.  Then for every restricted parabolic subgroup $P$, the restriction of $\phi$ to $R_u(P)$ is ${\rm log}_P$.
  \item[(iii)] Any two logarithmic Springer maps induce the same map from $\UU^{(1)}_G$ to $\NN^{(1)}_G$.
 \end{itemize}
\end{prop}

\begin{proof}
 First assume that $G$ is simple and $p$ is separably good for $G$.  Part (ii) follows from \cite[Prop.\ 2.1]{sobaje}.  For part (i), let $\varphi\colon \NN_G\to \UU_G$ be a $G$-equivariant isomorphism of varieties as in \cite[Thm.\ 4.1]{sobaje}.  Fix a maximal unipotent subgroup $U$ of $G$.  By \cite[Thm.\ 1.1]{sobaje}, $d\varphi\colon \Lie(U)\to \Lie(U)$ is a scalar multiple of the identity.  Condition (1) of \cite[Thm.\ 4.1]{sobaje} implies that this scalar is 1, so $d\varphi$ is the identity map.
 Let $1\neq u\in \UU_G^{(1)}$ and set $X= \varphi^{-1}(u)$.  Then $X\in \NN_G^{(1)}$ by Remark~\ref{rem:1-preserving}, so by \cite[Cor.\ 4.3(1)]{sobaje}, $\varphi$ gives an isomorphism from $kX$ onto a 1-dimensional unipotent subgroup $U'$ of $G$ which is contained in a good $A_1$ subgroup of $G$.  By construction, $U'= \mathcal{U}(u)$.  Since $d\varphi$ is the identity map, $X$ belongs to $\Lie(\mathcal{U}(u))$, so $\varphi$ gives an isomorphism of algebraic groups from $\Lie(\mathcal{U}(u))$ to $\mathcal{U}(u)$.  It follows that $\varphi^{-1}$ is a logarithmic Springer map for $G$, so (i) is proved.
 
 Now let $1 \ne u\in \UU_G^{(1)}$.  Choose a good $A_1$ overgroup $A$ of $u$ in $G$.  Choose a maximal torus $T$ of $A$ such that $T$ normalises $\mathcal{U}(u)$.  Definition~\ref{defn:logarithmic} and the $T$-equivariance of $\varphi^{-1}$ imply that the map from $\mathcal{U}(u)$ to $\Lie(\mathcal{U}(u))$ induced by $\varphi^{-1}$ does not depend on the choice of $\varphi^{-1}$.  This proves part (iii).
 
 The result now follows for arbitrary connected reductive $G$ using Remark~\ref{rem:springer_homeom}.
\end{proof}

\begin{rem}
 If $p> h(G)$ then the map ${\rm log}$ from Example~\ref{ex:sat_compatible} is a logarithmic Springer map (see \cite[Thm.\ 3]{serre1} and \cite[Rem.\ 27]{mcninch2}).  In this case any Borel subgroup of $G$ is a restricted parabolic, so the restriction of any logarithmic Springer map for $G$ to $R_u(B)$ is ${\rm log}_B$ by Proposition~\ref{prop:logarithmic_springer}(ii).  Hence ${\rm log}$ is the unique logarithmic Springer map for $G$.
\end{rem}

\begin{rem}
 We saw above that the condition on $\phi$ in Definition~\ref{defn:logarithmic} implies part (ii) of Proposition~\ref{prop:logarithmic_springer}.  Sobaje observes at the beginning of \cite[\S  2]{sobaje} that the converse also holds.  The reason is that every $1 \ne u\in \UU_G^{(1)}$ belongs to $R_u(P)$ for some restricted parabolic subgroup $P$ of $G$: this follows from \cite[Thm.~2.4]{carlsonlinnakano}.  We also deduce that the restriction of ${\rm log}_P$ to $\mathcal{U}(u)$ is $\phi_u$ for every restricted parabolic subgroup $P$ of $G$ and every $u\in R_u(P)$ such that $u$ has order $p$.
\end{rem}

\begin{cor}
\label{cor:compatible_good}
 Let $\phi$ be a logarithmic Springer map for $G$.  Then for any $A_1$ subgroup $A$ of $G$, $A$ is good for $G$ if and only if $\phi$ is $A$-compatible.
\end{cor}

\begin{proof}
 Suppose $A$ is good.  Let $1\neq u\in \UU_A$.  Then $\mathcal{U}(u)\subseteq A$ and $\phi(\mathcal{U}(u))= \Lie(\mathcal{U}(u))\subseteq \Lie(A)$, so $\phi$ is $A$-compatible.  Conversely, suppose $\phi$ is $A$-compatible.  Let $1\neq u\in \UU_A$ and set $X= \phi(u)\in \Lie(\mathcal{U}(u))$.  Now $X\in \Lie(A)$ by the $A$-compatibility, so $kX\subseteq \Lie(A)$.  Hence $\mathcal{U}(u)= \phi^{-1}(kX)\subseteq A$ by the $A$-compatibility.   We deduce from Lemma~\ref{lem:U_u_overgroup} that $A$ is good for $G$.
\end{proof}

We now recall the other types of $A_1$ subgroup that we need, namely optimal and sub-principal $A_1$ subgroups.  These were introduced by McNinch in \cite{mcninch2} and \cite{mcninch1}.

\begin{defn}
 We call a homomorphism $\beta\colon \SL_2\to G$ \emph{optimal} if there is a maximal torus $T$ of $\SL_2$ such that the restriction $\lambda$ of $\beta$ to $T\iso \Gm$ is a cocharacter associated in $G$ to some nilpotent $0\neq X\in {\rm Im}(d\beta)$.  We call an $A_1$ subgroup of $G$ \emph{optimal} if it is the image of an optimal homomorphism.
\end{defn}

\begin{rem}
 This is equivalent to the definition in \cite[\S  1]{mcninch2}: for it is clear that if $T$ is the standard maximal torus of $\SL_2$ and $\lambda$ is associated to some nilpotent $0\neq X\in \Lie(\SL_2)$ then $X$ is a scalar multiple of $d\phi\left(\left(\begin{array}{cc} 0 & 1 \\ 0 & 0 \end{array}\right)\right)$.
\end{rem}

\begin{defn}
 Fix a Springer map $\phi$ for $G$.  We call a homomorphism $\beta\colon \SL_2\to G$ \emph{sub-principal} if there is a maximal torus $T$ of $\SL_2$ such that the restriction $\lambda$ of $\beta$ to $T\iso \Gm$ is a cocharacter associated in $G$ to some nilpotent $0\neq X\in {\rm Im}(d\beta)$ and $\phi^{-1}(X)$ is $G$-conjugate to an element of ${\rm Im}(\beta)$.  Note that the latter condition does not depend on the choice of $\phi$, by Lemma~\ref{rem:conj_class_map}.  We call an $A_1$ subgroup of $G$ \emph{sub-principal} if it is the image of a sub-principal homomorphism.
\end{defn}

The next result implies Theorem~\ref{thm:good_equivalences_intro}.

\begin{thm}
\label{thm:good_equivalences}
 Let $A$ be an $A_1$ subgroup of $G$.  Let $\phi$ be a logarithmic Springer  map for $G$.  The following conditions are equivalent.
\begin{itemize}
 \item[(i)] $A$ is sub-principal.
 \item[(ii)] $A$ is optimal.
 \item[(iii)] There exist $u\in \UU_G^{(1)}$ and $\lambda\in Y(A)$ such that $\lambda\in \Omega^a_{G, \phi}(u)$.
 \item[(iv)] $A$ is good.
\end{itemize}
\end{thm}

\begin{proof}
 The implication (i) $\implies$ (ii) is immediate from the definitions, and (iii) $\implies$ (iv) follows from Lemma~\ref{lem:good_crit}.  If $A$ is optimal then there exist $0\neq X\in \NN_A$ and $\lambda\in Y(A)$ such that $\lambda\in \Omega^a_G(X)$.  Then $\lambda\in \Omega^a_{G,\phi}(u')$, where $u'= \phi^{-1}(X)$, and $u'$ has order $p$ by Lemma~\ref{lem:p_preserve}.  Hence (ii) $\implies$ (iii).
 
 By \cite[Rem.\ 21]{mcninch1}, there exists at least one sub-principal $A_1$ subgroup $A$ of $G$ such that $u\in A$, and $A$ is good by the arguments above.  It is clear from the definition that any $C_G(u)$-conjugate of $A$ is sub-principal.  Since $C_G(u)$ acts transitively on $\mathscr{A}(u)$ (Theorem~\ref{thm:seitz1}(ii)), it follows that any good $A_1$ subgroup of $G$ that contains $u$ is sub-principal.  This shows that (iv) $\implies$ (i).  Hence (i)--(iv) are all equivalent.
 \end{proof}

\begin{rem}
 If the equivalent conditions from Theorem~\ref{thm:good_equivalences} hold then there exist $\lambda\in Y(A)$ and $0\neq X\in \NN_A$ such that $\lambda\in \Omega^a_G(X)$.  Then $\lambda\in \Omega^a_G(u)$, where $u= \phi^{-1}(X)$, which belongs to $A$ by Corollary~\ref{cor:compatible_good}.  Hence we can take the element $u$ from Theorem~\ref{thm:good_equivalences}(iii) to belong to $A$ if we wish.
\end{rem}

\begin{rem}
 It is implicit in the discussion in \cite[\S  1]{mcninch1} that a sub-principal $A_1$ subgroup of $G$ is good.  McNinch also proved that goodness and optimality are equivalent for $A_1$ subgroups under the extra assumption that $G$ is strongly standard (see \cite[Prop.\ 53]{mcninch2}).
\end{rem}

\begin{prop}
\label{prop:good_levi_ascent}
 Let $L$ be a Levi subgroup of $G$ and let $A$ be a good $A_1$ subgroup of $L$.  Then $A$ is a good $A_1$ subgroup of $G$.
\end{prop}

\begin{proof}
 Since $p$ is good for $G$, $p$ is good for $L$.  By Theorem~\ref{thm:good_equivalences}, $A$ is optimal in $L$, so there exist $0\neq X\in \NN_A$ and $\lambda\in Y(A)$ such that $\lambda\in \Omega^a_L(X)$.  Lemma~\ref{lem:RFlevis_Lie} implies that $\lambda\in \Omega^a_G(X)$, so $A$ is optimal in $G$.  Hence $A$ is a good $A_1$ subgroup of $G$ by Theorem~\ref{thm:good_equivalences}.
\end{proof}

\begin{cor}
\label{cor:same_Levi}
 Let $A$ be a good $A_1$ subgroup of $G$ and let $1\neq u\in \UU_A$.  Then there is a Levi subgroup $L$ of $G$ such that $A\subseteq L$ and $u$ is a distinguished unipotent element of $L$.
\end{cor}

\begin{proof}
 Pick a Levi subgroup $L'$ of $G$ such that $u$ is a distinguished unipotent element of $L'$.  By Theorem~\ref{thm:seitz1}(i) we can choose a good $A_1$ subgroup $A'$ of $L'$ such that $u\in A'$.  Now $A'$ is a good $A_1$ subgroup of $G$ by Proposition~\ref{prop:good_levi_ascent}, so there exists $g\in C_G(u)$ such that $gA'g^{-1}= A$ (Theorem~\ref{thm:seitz1}(ii)).  Then $A\subseteq L$, where $L:= gL'g^{-1}$.  Clearly, $L$ is a Levi subgroup of $G$ and $u$ is a distinguished unipotent element of $L$.
\end{proof}

\begin{cor}
\label{cor:logarithmic_levi}
 Let $\phi$ be a logarithmic Springer map for $G$ and let $L$ be a Levi subgroup of $G$.  Then $\phi_L$ is a logarithmic Springer map for $L$.
\end{cor}

\begin{proof}
 Let $1 \ne u\in \UU_L^{(1)}$.  Choose a good $A_1$ overgroup $A$ of $u$ in $L$.  Then $A$ is good for $G$ by Proposition~\ref{prop:good_levi_ascent}, so $\mathcal{U}(u)\subseteq A$.  We see that $\mathcal{U}(u)$ is both the unique 1-dimensional overgroup of $u$ that is contained in a good $A_1$ subgroup of $L$, and the unique 1-dimensional overgroup of $u$ that is contained in a good $A_1$ subgroup of $G$.  The result now follows from the definition of a logarithmic Springer map.
\end{proof}

\begin{lem}
\label{lem:good_ascent}
 Let $H$ be a connected reductive subgroup of $G$, and assume $p$ is good for $H$.  Let $u\in \UU_H^{(1)}$  such that $u$ is distinguished in $G$. Let $A \in \AA_H(u)$. Suppose there is a  Springer map $\phi$ for $H$ such that $\phi(u)$ is a distinguished element of $\frakg$.  Then $A$ is good for $G$.
\end{lem}

\begin{proof}
 By Theorem~\ref{thm:good_equivalences}, $A$ is a sub-principal $A_1$ subgroup of $H$, so there exist $\lambda\in Y(A)$ and $0\neq X\in \NN_A$ such that $\lambda$ is associated to $X$ in $H$ and $\phi(u)$ is $H$-conjugate to $X$.  Since by hypothesis $\phi(u)$ is a distinguished element of $\frakg$, $X$ is also a distinguished element of $\frakg$.  Lemma~\ref{lem:RF_Lie} implies that $\lambda$ is associated to $X$ in $G$.  It follows that $A$ is an optimal $A_1$ subgroup of $G$, so $A$ is a good $A_1$ subgroup of $G$ by Theorem~\ref{thm:good_equivalences}.
\end{proof}

\begin{rem}
 Let $H$ be a connected reductive subgroup of $G$ and assume $p$ is good for $H$.  Suppose $H$ is Springer-compatible.  Let $A$ be an $A_1$ subgroup of $H$ containing a distinguished unipotent element $u$ of $G$.  Let $\phi$ be the restriction to $H$ of any $H$-compatible Springer map for $G$.  Then $\phi(u)$ is a distinguished element of $\frakg$ by Remark~\ref{rem:distinguished_pres}, so the hypotheses of Lemma~\ref{lem:good_ascent} hold.  Hence if $A$ is good for $H$ then $A$ is good for $G$.
\end{rem}

The following relates  the set of cocharacters of $G$ that are associated to some $1 \ne u\in \UU_G^{(1)}$ to those stemming from good $A_1$ overgroups of $u$ in $G$.

\begin{cor}
\label{cor:local_associated}
 Let $1 \ne u\in \UU_G^{(1)}$.  Let $\phi$ be a logarithmic Springer map for $G$.  We have a disjoint union
\[
\Omega^a_{G, \phi}(u) = \dot{\bigcup_{A \in \AA(u)}} \Omega^a_{A, \phi_A}(u),
\]
where $\phi_A$ denotes the restriction of $\phi$ to $A$.
\end{cor}

\begin{proof}
 Note that it makes sense to speak of the restriction of $\phi$ to a good $A_1$ subgroup $A$ of $G$, by Corollary~\ref{cor:compatible_good}.  We first prove that the union above is disjoint.
% The argument is more less identical to the proof of  Theorem \ref{thm:seitz1}(ii) above.
 Let $A, \tilde{A} \in \AA(u)$ and suppose there exists some $\lambda \in \Omega^a_{A, \phi_A}(u) \cap \Omega^a_{\tilde{A}, \phi_{\tilde{A}}}(u)$. Then $A$ and $\tilde{A}$ share the common Borel subgroup $\lambda(\BBG_m) \mathcal{U}(u)$. It follows from \cite[Lem.~2.4]{londmartin} that $A = \tilde{A}$.
 
 Let $A\in \AA(u)$ and let $\lambda\in \Omega^a_{A, \phi_A}(u)$.  By Corollary~\ref{cor:same_Levi} there is a Levi subgroup $L$ of $G$ such that $A\subseteq L$ and $u$ is a distinguished unipotent element of $L$.  It follows from Lemma~\ref{lem:RF} (applied to the inclusion $A\subseteq L$) and Lemma~\ref{lem:RFlevis} that $\lambda\in \Omega^a_{G, \phi}(u)$.  Hence $\Omega^a_{G, \phi}(u) \supseteq \dot{\bigcup}_{A \in \AA(u)} \Omega^a_{A, \phi_A}(u)$.  Since $\AA(u)\neq \varnothing$ (Theorem~\ref{thm:seitz1}(i)), $C_G(u)$ acts transitively on both $\mathscr{A}(u)$ (Theorem~\ref{thm:seitz1}(ii)) and $\Omega^a_{G, \phi}(u)$ (Proposition~\ref{prop:cochar1}(ii)), we see that the reverse inclusion follows.
\end{proof}

\section{Proofs of Theorems \ref{thm:dist-orderp} and \ref{thm:korhonen-bad}--\ref{thm:korhonen-bad-finite} }
\label{sec:proofs}

Armed with the results from above,  we prove Theorems \ref{thm:dist-orderp} and \ref{thm:dist-orderp-finite} simultaneously. 

\begin{proof}[Proof of Theorems \ref{thm:dist-orderp} and \ref{thm:dist-orderp-finite}]
 We may assume that $G$ is semisimple, since any unipotent element of $G$ is contained in the derived subgroup $\DD G^\circ$. Likewise, we may also assume that $H$ is connected and semisimple, as any unipotent element of $H^\circ$ is contained in the derived subgroup $\DD H^\circ$, and $H$ is $G$-ir if $\DD H^\circ$ is.
 Let $u \in \UU_H^{(1)}$ be distinguished in $G$.
 
 First suppose $p$ is bad for $H$.  
 If $p> 2$ then $H$ admits a simple component $H'$ of exceptional type.  If $u\in H$ is a distinguished unipotent element of $G$ then the projection $u'$ of $u$ onto $H'$ is a distinguished unipotent element of $H'$, so $p= 3$ and $H'$ is of type $G_2$, by Lemma~\ref{lem:pst42}.  But this is impossible by Lemma~\ref{lem:g2p=3} since $p$ is good for $G$.  Hence $p= 2$.  It follows that each simple component of $G$ is of type $A$.  Now distinguished unipotent elements are regular in type $A$, so $u$ is a regular element of $G$.  It follows from \cite[Thm.~1.1]{BMR:regular} (resp., \cite[Thm.~1.3]{BMR:regular}) that $H$ (resp., $H_\sigma$) is $G$-ir.
 
  Hence we can assume that $p$ is good for $H$. By Theorem~\ref{thm:seitz1}(i) (resp., Theorem~\ref{thm:seitz2}(i)) there is a good $A_1$ subgroup (resp., good $\sigma$-stable $A_1$ subgroup) $A$ of $H$ such that $u\in A$.  By Lemma~\ref{lem:good_ascent} and hypothesis $(\dagger)$, $A$ is a good $A_1$ subgroup of $G$.  Hence $A$ (resp., $A_\sigma$) is $G$-cr by Theorem~\ref{thm:seitz1}(iii)) (resp., Theorem~\ref{thm:seitz2})(iii)), so $A$ (resp., $A_\sigma$) is $G$-ir by Lemma~\ref{lem:GcrGir;dist}.  We conclude that $H$ (resp., $H_\sigma$) is $G$-ir. 
\end{proof}

\begin{rem}
\label{rem:daggerless}
 If we remove hypothesis $(\dagger)$ from Lemma~\ref{lem:good_ascent}, Theorem~\ref{thm:dist-orderp}, etc., then our arguments break down.  For instance, let $G= \SL_2\times \SL_2$, $q$ and $H_q$ be as in Example~\ref{ex:no_good}.  Let $u$ be any unipotent element of $H_q$ such that the projection of $u$ onto each $\SL_2$-factor of $H_q$ is non-trivial; then $u$ is distinguished (in fact, regular) in $G$.  It is easy to see that $H_q$ is not a good $A_1$ subgroup of $G$ and there does not exist $\lambda\in Y(H_q)$ such that $\lambda$ is associated to $u$ in $G$; in particular, the conclusion of Lemma~\ref{lem:good_ascent} does not hold for $H_q$.  Of course Theorem~\ref{thm:korhonen} still applies, alternately so does Theorem \ref{thm:dist}, so $H_q$ is $G$-ir.
\end{rem}

As a consequence of Theorems \ref{thm:dist-orderp} and \ref{thm:dist-orderp-finite} we obtain the following.

\begin{cor}
	\label{cor:gcrovergroups}
	Let $G$ be a connected reductive group. Suppose $p$ is good for $G$.
 Let $\sigma$ be $\id_G$ or a $q$-Frobenius endomorphism of $G$. Let $u \in G_\sigma$ be unipotent of order $p$.
Suppose $u$ is distinguished in the $\sigma$-stable Levi subgroup $L$ of $G$ (see Remark \ref{rem:sigmastablelevi}(ii)). Let $H$ be a $\sigma$-stable connected reductive  subgroup of $L$ containing $u$, and suppose there is a Springer map $\phi$ for $L$ such that $\phi(u)$ is a distinguished element of $\Lie(L)$.  Then $H_\sigma$ is $G$-completely reducible.
\end{cor}

\begin{proof}
As $p$ is also good for $L$ (see \S \ref{sec:goodprimes}),
 it follows from Theorem \ref{thm:dist-orderp} (resp.~\ref{thm:dist-orderp-finite}) applied to $L$ that $H_\sigma$ is $L$-ir and so is $L$-cr. Thus, $H_\sigma$ is  $G$-cr, by 
 \cite[Prop.~3.2]{serre2}. 
\end{proof}

\begin{rem}
	\label{rem:simple_reduction}
	In the setting of Theorem~\ref{thm:korhonen} the following argument allows to reduce the case when  $G$ is connected reductive to the simple case.
	As in the proof of Theorem~\ref{thm:dist-orderp} above, we can assume that $G$ is semisimple.
	Let $G_1,\ldots, G_r$ be the simple factors of $G$.  Multiplication gives an isogeny from $G_1\times\cdots \times G_r$ to $G$.  Thus, by \cite[Lem.~2.12(ii)(b)]{BMR} and \cite[\S 4.3]{Jantzen},
	we can replace $G$ with $G_1\times\cdots \times G_r$, so we can assume $G$ is the product of its simple factors.  
	Finally, thanks to \cite[Lem.~2.12(i)]{BMR}, \cite[\S 4.3]{Jantzen}, we can reduce to the case when $G$ is simple.  
\end{rem}

Finally, we address Theorems \ref{thm:korhonen-bad} and \ref{thm:korhonen-bad-finite}.

\begin{proof}[Proof of Theorems \ref{thm:korhonen-bad} and \ref{thm:korhonen-bad-finite}]
By Theorems \ref{thm:dist-orderp} and \ref{thm:dist-orderp-finite}, the only cases we need to consider are when $p$ is bad for $G$.
If $G$ is classical, then we are in the situation of Lemma~\ref{lem:pst41} and Example~\ref{ex:c2}, so we are done.

We are left to consider the case when 
$G$ is of exceptional type. 
Then owing to Lemma \ref{lem:pst42}, $G$ is of type $G_2$ and $p = 3$. 
There is no harm in assuming that $H$ is semisimple.
It follows from Example \ref{ex:g2} that $H$ is $G$-ir.
Thus Theorem~\ref{thm:korhonen-bad} follows.
So consider the setting of 
Theorem \ref{thm:korhonen-bad-finite} when $\sigma|_H$ is a $q$-Frobenius endomorphism of $H$ in this case. By Corollary \ref{cor:g2},
$u$ belongs to the subregular class of $G_2$.
It follows from the proof of Lemma \ref{lem:pst42} in \cite{PST} that $u$ is contained in a 
$\sigma$-stable
maximal rank subgroup of $G$ of type $A_1 \widetilde A_1$ and this type is unique.
Since $H$ is proper and semisimple, $H \subseteq M$, where 
$M$ is a $\sigma$-stable maximal rank subgroup of $G$ of type $A_1 \widetilde A_1$.
Since $p$ is good for $H$, there is a $\sigma$-stable subgroup $A$ of $H$ of type $A_1$ containing $u$, by Theorem \ref{thm:A1pgood}.
Thus $A \subseteq H \subseteq M$.
Since $u$ is also distinguished in $M$ and $p = 3$ is good for $M$, 
Theorem \ref{thm:dist-orderp-finite} shows that $A_\sigma$ is $M$-ir. 
Note that $M$ is the centralizer of a semisimple element of $G$ of order $2$
(by Deriziotis' Criterion, see \cite[2.3]{Deriziotis}).
Since $A_\sigma$ is $M$-cr, it is $G$-cr, owing to \cite[Cor.~3.21]{BMR}. Once again, by Lemma \ref{lem:GcrGir;dist}, $A_\sigma$ is $G$-ir and so is $H_\sigma$.
Theorem \ref{thm:korhonen-bad-finite} follows.
\end{proof}

\begin{rem}
    \label{rem:higherorder}
     In \cite[\S 7]{korhonen}, Korhonen gives counterexamples to Theorem \ref{thm:korhonen} when the order of the distinguished unipotent element of $G$ is greater than $p$ (even when $p$ is good for $G$ \cite[Prop.~7.1]{korhonen}). Theorem \ref{thm:dist} implies that this can only happen when $p < a(G)$.
	For instances of overgroups of distinguished unipotent elements of $G$ of order greater than $p$ for $p \ge a(G)$ (and $p$ good for $G$), so that Theorem \ref{thm:dist} applies, see Examples \ref{ex:c4e6} and \ref{ex:e8}.
\end{rem}

\begin{rem}
	\label{rem:maximalrank}
	In view of Remark \ref{rem:higherorder}, it is natural to ask for instances  of $G$, $u$ and $H$ when the conclusion of Theorem \ref{thm:dist} holds even when $p < a(G)$ but $p$ is still good for $G$. If $p$ is good for $G$ and 
	$G$ is simple classical, non-regular distinguished unipotent elements always belong to a maximal rank semisimple subgroup $H$ of $G$, by \cite[Prop.~3.1, Prop.~3.2]{testerman}. For $G$ simple of exceptional type this is also the case in almost all instances of non-regular distinguished unipotent elements, see \cite[Lem.~2.1]{testerman}. 
	Each such $H$ is obviously $G$-irreducible. This is independent of $p$ of course and thus applies in particular when  $p< a(G)$.  
	For instance, let $G$ be of type $E_7$, $p = 5$, and suppose $u$ belongs to the distinguished class $E_7(a_3)$ (resp.~$E_7(a_4)$, $E_7(a_5)$). Then $\hgt_J(\rho) = 9$ (resp.~$7$, $5$), so $u$ has order $5^2$, by Lemma \ref{lem:order} in each case.
	Since $u$ does not have order $5$, Theorem \ref{thm:korhonen} does not apply, and since $5 < 8 = a(G)$ neither does Theorem \ref{thm:dist}. 
	Nevertheless, in each case $u$ is contained in a maximal rank subgroup $H$ of type $A_1D_6$, see \cite[p.~52]{testerman}, and each such $H$ is $G$-ir.
\end{rem}

We close the section with several additional higher order examples in good characteristic when Theorem \ref{thm:korhonen} does not apply but Theorem \ref{thm:dist} does.

\begin{ex}
	\label{ex:c4e6}
	Let $G$ be of type $E_6$. Suppose $p$ is good for $G$.
	In \cite[Lem.\ 2.7]{testerman}, Testerman exhibits the existence of a simple subgroup $H$ of $G$ of type $C_4$  whose regular unipotent class belongs to the subregular class $E_6(a_1)$ of $G$. Let $u$ be regular unipotent in $H$. 
	For $p = 7$, the order of $u$ is $7^2$, by Lemma \ref{lem:order}, so Theorem \ref{thm:korhonen} can't be invoked to say anything about $H$. However, for $p = 7 = a(G)$, we infer from Theorem \ref{thm:dist} that $H$ is $G$-ir. 
\end{ex}

\begin{ex}
	\label{ex:e8}
	Let $G$ be of type $E_8$. Suppose $p = 11$.
	Let $u$ be in the distinguished class $E_8(a_3)$ (resp.~$E_8(a_4)$, $E_8(b_4)$, $E_8(a_5)$, or $E_8(b_5)$). From the corresponding weighted Dynkin diagram corresponding to $u$ we get $\hgt_J(\rho) = 17$ (resp.~$14$, $13$, $11$, or $11$), see \cite[p.~177]{carter:book}. It follows from Lemma \ref{lem:order} that in each of these instances $u$ has order $11^2$. So we can't appeal to Theorem \ref{thm:korhonen} to deduce anything about reductive overgroups of $u$. But as $11 = p \ge a(G) = 9$,  
	Theorem \ref{thm:dist} applies and allows us to conclude that each such overgroup is $G$-ir. For example, in each instance above, $u$ is contained in a maximal rank subgroup $H$ of $G$ of type $A_1E_7$ or $D_8$, see \cite[p.~52]{testerman}.
\end{ex}

\bigskip
%%%%%%%%%%%%%%%%%%%%%%%%%%%%%%%%%%%%%%%%%%%%%%%%%%%%%%%%%%%%%%%%%%%%%%
%%%%%%%%%%%%% Acknowledgements
%%%%%%%%%%%%%%%%%%%%%%%%%%%%%%%%%%%%%%%%%%%%%%%%%%%%%%%%%%%%%%%%%%%%%%

%\bigskip

\noindent {\bf Acknowledgments}:
We are grateful to M.~Korhonen and D.~Testerman for helpful comments on an earlier version of the manuscript, and to A.~Thomas for providing the $G_2$ example in Example \ref{ex:no_good}. We thank the referee for a number of comments clarifying some points. 
 The research of this work was supported in part by
the DFG (Grant \#RO 1072/22-1 (project number: 498503969) to G.~R\"ohrle). Some of this work was completed during a visit to the Mathematisches Forschungsinstitut Oberwolfach: we thank them for their support.  
For the purpose of open access, the authors have applied a Creative Commons Attribution (CC BY) licence to any Author Accepted Manuscript version arising from this submission.

%\bigskip
%%%%%%%%%%%%%%%%%%%%%%%%%%%%%%%%%%%%%%%%%%%%%%%%%%%%%%%%%%%%%%%%%%%%%%
%%%%%%%%%%%%% bibliography
%%%%%%%%%%%%%%%%%%%%%%%%%%%%%%%%%%%%%%%%%%%%%%%%%%%%%%%%%%%%%%%%%%%%%%

\end{document}